\documentclass[10pt,reqno]{article}
\usepackage[english]{babel}
\usepackage{graphicx}
\usepackage[centertags]{amsmath}
\usepackage{amsfonts,amssymb,amsthm}
\usepackage{hyperref}
\usepackage[utf8]{inputenc} 
\usepackage{color}
\usepackage{anysize}
\marginsize{3.0cm}{3.0cm}{2.2cm}{2.2cm}

\let\OLDthebibliography\thebibliography
\renewcommand\thebibliography[1]{
  \OLDthebibliography{#1}
  \setlength{\parskip}{0pt}
  \setlength{\itemsep}{0pt plus 0.3ex} }

\numberwithin{equation}{section}

\theoremstyle{plain}
\newtheorem{theorem}{Theorem}[section]
\newtheorem{proposition}[theorem]{Proposition}
\newtheorem{lemma}[theorem]{Lemma}

\theoremstyle{definition}
\newtheorem{remarkbase}[theorem]{Remark}
\newenvironment{remark}{\pushQED{\qed}\remarkbase}{\popQED\endremarkbase}

\def\notina[#1]#2{\begingroup\def\thefootnote{\fnsymbol{footnote}}\footnote[#1]{#2}\endgroup}

\newcommand{\N}{{\mathbb N}}
\newcommand{\R}{{\mathbb R}}
\newcommand{\C}{{\mathbb C}}
\newcommand{\Z}{\mathbb Z}

\newcommand{\T}{{\mathbb T}}

\newcommand{\mA}{\mathcal{A}}
\newcommand{\mB}{\mathcal{B}}
\newcommand{\mC}{\mathcal{C}}
\newcommand{\mD}{\mathcal{D}}
\newcommand{\mE}{\mathcal{E}}
\newcommand{\mF}{\mathcal{F}}
\newcommand{\mG}{\mathcal{G}}

\newcommand{\mK}{\mathcal{K}}

\newcommand{\mM}{\mathcal{M}}
\newcommand{\mP}{\mathcal{P}}
\newcommand{\mR}{\mathcal{R}}

\newcommand{\mT}{\mathcal{T}}

\renewcommand{\a}{\alpha}
\renewcommand{\b}{\beta}
\newcommand{\g}{\gamma}
\renewcommand{\d}{\delta}

\newcommand{\e}{\varepsilon}
\newcommand{\ph}{\varphi}
\newcommand{\lm}{\lambda}

\newcommand{\Om}{\Omega}
\newcommand{\om}{\omega}

\newcommand{\s}{\sigma}

\renewcommand{\th}{\vartheta}

\newcommand{\gr}{\nabla}
\newcommand{\grad}{\nabla}

\newcommand{\nth}{\vartheta}

\newcommand{\la}{\langle}
\newcommand{\ra}{\rangle}

\newcommand{\pa}{\partial}
\newcommand{\Lm}{\Lambda}

\newcommand{\rag}{\rho} 

\title{Longer lifespan for many solutions of the Kirchhoff equation}

\author{\small{Pietro Baldi, Emanuele Haus}}

\date{} 

\begin{document}

\maketitle


\begin{small} 
\noindent
\emph{Abstract}.
We consider the Kirchhoff equation
\begin{equation*}
\pa_{tt} u - \Delta u \Big( 1 + \int_{\T^d} |\gr u|^2 \Big) = 0
\end{equation*}
on the $d$-dimensional torus $\T^d$, 
and its Cauchy problem with initial data 
$u(0,x)$, $\pa_t u(0,x)$ 
of size $\e$ in Sobolev class.
The effective equation for the dynamics at the quintic order, 
obtained in previous papers by quasilinear normal form, 
contains resonances corresponding to nontrivial terms in the energy estimates. 
Such resonances cannot be avoided by tuning external parameters 
(simply because the Kirchhoff equation does not contain parameters). 

In this paper we introduce nonresonance conditions on the initial data of the Cauchy problem
and prove a lower bound $\e^{-6}$ for the lifespan of the corresponding solutions 
(the standard local theory gives $\e^{-2}$, 
and the normal form for the cubic terms gives $\e^{-4}$).
The proof relies on the fact that,
under these nonresonance conditions, 
the growth rate of the ``superactions'' of the effective equations 
on large time intervals is smaller (by a factor $\e^2$) than its a priori estimate 
based on the normal form for the cubic terms.
The set of initial data satisfying such nonresonance conditions
contains several nontrivial examples 
that are discussed in the paper. 

\medskip


\end{small}

\begin{footnotesize}
\tableofcontents
\end{footnotesize}

\section{Introduction}  
\label{sec:1}

We consider the Cauchy problem for the Kirchhoff equation 
on the $d$-dimensional torus $\T^d$, $\T := \R / 2\pi \Z$ 
(periodic boundary conditions)  
\begin{equation} \label{K1}
\pa_{tt} u - \Delta u \Big( 1 + \int_{\T^d} |\gr u|^2 \, dx \Big) = 0, 
\quad\text{where}\ u=u(t,x),\ x\in\T^d
\end{equation}
with initial data at time $t=0$
\begin{equation} \label{initial data 1} 
u(0,x) = a(x), \quad 
\pa_t u(0,x) = b(x).
\end{equation}
While it is known 
(Dickey \cite{Dickey 1969}, Arosio-Panizzi \cite{Arosio Panizzi 1996}) 
that such a Cauchy problem is locally wellposed 
for initial data $(a,b)$ in the Sobolev space $H^{\frac32}(\T^d, \R) \times H^{\frac12}(\T^d, \R)$, 
it is a still open problem whether the solutions of \eqref{K1}-\eqref{initial data 1} of any given Sobolev regularity are global in time or not. 
In particular, it is not even known if $C^\infty$ initial data of small amplitude produce solutions that are global in time (for initial data in analytic class, instead, global wellposedness is known since the work of Bernstein \cite{Bernstein 1940} in 1940).

As a consequence of the linear theory, one has a lower bound of $\e^{-2}$ for the lifespan of solutions corresponding to initial data of size $\e$. Since \eqref{K1} is a quasilinear wave equation, it is not a priori obvious that one can obtain better estimates. For instance, in the well-known example by Klainerman and Majda \cite{Klainerman Majda 1980} {\em all} space-periodic nontrivial solutions of size $\e$ blow up in a time of order $\e^{-2}$. For the Kirchhoff equation, however, the situation is more favorable: as we proved in \cite{K old}, after one step of quasilinear normal form, the only cubic terms that cannot be erased give no contribution to the time evolution of Sobolev norms; this allowed us to extend the lifespan of all solutions of small amplitude to $\e^{-4}$.

In the recent paper \cite{K mem}, we computed the second step of quasilinear normal form for the Kirchhoff equation and showed that there are resonant terms of degree five that cannot be erased and give a nontrivial contribution to the time evolution of Sobolev norms. 
Here we show that for a suitable set of ``nonresonant'' initial data the effect of these terms can be neglected on a longer timescale and the lifespan of the corresponding solution is at least $\e^{-6}$ 
(Theorem \ref{thm:main} below).

\bigskip

Equation \eqref{K1}, 
introduced by Kirchhoff \cite{Kirchhoff 1876} 
as a nonlinear model for vibrating strings and membranes,
belongs to the class of Hamiltonian PDEs, 
as it can be written as the system 
\begin{equation} \label{p1}
\begin{cases} 
\pa_t u = \gr_v H(u,v) = v, \\ 
\pa_t v = - \gr_u H(u,v) = \Delta u \Big( 1 + \int_{\T^d} |\gr u|^2 dx \Big),
\end{cases}
\end{equation}
where the Hamiltonian is
\begin{equation} \label{K2} 
H(u,v) = \frac12 \int_{\T^d} v^2 dx 
+ \frac12 \int_{\T^d} |\gr u|^2 dx 
+ \Big( \frac12 \int_{\T^d} |\gr u|^2 dx \Big)^2,
\end{equation} 
and $\gr_u H$, $\gr_v H$ are the gradients with respect to the 
real scalar product of $L^2(\T^d,\R)$.

When the Cauchy problem for a Hamiltonian PDE is set on a compact manifold 
(like $\T^d$), 
dispersion mechanisms that hold on $\R^d$ are 
not available, 
and the main tool to prove existence beyond the time 
of the standard local theory is the \emph{normal form} method. 
Important references on normal forms of Hamiltonian PDEs on compact manifolds
are the works of 
Kuksin, Kappeler, P\"oschel \cite{Kuksin Poschel 1996 Annals}, 
\cite{Kappeler Poschel 2003},
Bourgain \cite{Bourgain 2000 J Anal Math},
Bambusi, Gr\'ebert, Delort, Szeftel 
\cite{Bambusi 2003 CMP},
\cite{Bambusi Grebert 2006 Duke}, 
\cite{Delort Szeftel 2004 IMRN},
\cite{Bambusi Grebert Delort Szeftel 2007 CPAM}.
Some of the difficulties 
and achievements 
in this active research field 
regard the extension of the theory 
\begin{itemize}
\item[-]
to \emph{quasilinear} PDEs 
(see e.g.\ the results of Delort \cite{Delort 2009}, \cite{Delort 2012} on quasilinear Klein-Gordon equations, 
Craig-Sulem \cite{Craig Sulem 2016 BUMI}, 
Ifrim-Tataru \cite{Ifrim Tataru 2017},
Berti-Delort \cite{Berti Delort},  
Berti-Feola-Pusateri \cite{Berti Feola Pusateri} 
on water waves, 
Feola-Iandoli \cite{Feola Iandoli}, \cite{Feola Iandoli 2020}
on quasilinear NLS and abstract methods), 

\item[-] 
to \emph{resonant} equations without the help of external parameters
(see e.g.\ Bourgain \cite{Bourgain 2000 J Anal Math} and Buckmaster-Germain-Hani-Shatah \cite {BGHS} on NLS with random data,
Berti-Feola-Pusateri \cite{Berti Feola Pusateri}, \cite{Berti Feola Pusateri nota} 
on pure gravity water waves, 
Bernier-Faou-Gr\'ebert \cite{Bernier Faou Grebert Rational} on resonant NLS
with rational normal forms).  
\end{itemize}

The Kirchhoff equation \eqref{K1}, 
despite its simple structure, 
contains these difficulties: 
\begin{itemize} 
\item it is a \emph{quasilinear} PDE, 
because the nonlinear term $\Delta u \int |\grad u|^2$ 
has the same order of derivatives as the linear 
part of the equation; 

\item it is a \emph{resonant} equation: 
the linear frequencies of oscillation, 
namely the eigenvalues of the linear wave $\pa_{tt} - \Delta$, 
are square roots $|k| = \sqrt{k_1^2 + \ldots + k_d^2}$,
$k \in \Z^d$, of natural numbers, and therefore equations like $|k| + |j| - |\ell| = 0$ and similar,
which one encounters along a normal form procedure, have infinitely many nontrivial solutions; 

\item there are \emph{no external parameters} that could help to avoid the resonances;

\item in dimension $d \geq 2$, 
after the first step of normal form, 
the dominant term of the remaining resonant nonlinearity 
is \emph{not completely integrable}, namely it does not depend only on actions
(therefore the method of rational normal forms of \cite{Bernier Faou Grebert Rational}
does not directly apply to \eqref{K1}). 
\end{itemize}

The quasilinear normal form 
performed in \cite{K old}-\cite{K mem} 
(summarized in the appendix below), 
which is particularly simple because of the 
``already paralinearized'' structure of the Kirchhoff equation, 
overcomes the problem that the standard Birkhoff normal form construction, 
even at its first step, gives unbounded transformations. 
However, employing the quasilinear normal form to deduce 
a longer existence time for the Cauchy problem
presents all the other mentioned difficulties. 
To bypass them, 
in Theorem \ref{thm:main} we impose some nonresonance conditions 
on the ``superactions'' (see \eqref{intro.nonres}) on the initial data. 
On the other hand, we are far from proving an existence time of order $\e^{-6}$ 
for all, or even almost all, small initial conditions. 

\medskip

We refer the reader to Section 1.2 of our previous paper \cite{K old} 
regarding other properties of the Kirchhoff equation (reversibility, momenta, invariant subspaces), 
and to Section 1.3 of \cite{K old} 
for more references to the related literature, 
including some rich surveys.

\subsection{Main result}
\label{sec:main result}

On the torus $\T^d$, it is not restrictive to assume that 
both the initial data $a(x), b(x)$ and the unknown function $u(t,x)$ have zero average 
in the space variable $x$
(because the space average and the zero-mean component
of any $a,b,u$ satisfy two uncoupled Cauchy problems; 
the problem for the averages is elementary). 

For any real $s \geq 0$, we consider the Sobolev space of zero-mean functions
\begin{align} \label{def:Hs}
H^s_0(\T^d,\C) := \Big\{ u(x) = \sum_{j \in \Z^d \setminus \{ 0 \} } u_j e^{ij\cdot x} : 
u_j \in \C, \ 
\| u \|_s < \infty \Big\}, 
\qquad
\| u \|_s^2 := \sum_{j \neq 0} |u_j|^2 |j|^{2s},
\end{align}
and its subspace of real-valued functions
\begin{equation*} 
H^s_0(\T^d,\R) := \{ u \in H^s_0(\T^d,\C) : u(x) \in \R \}.
\end{equation*}
Define 
\begin{equation} \label{def m1}
m_1 := 1 \ \ \text{if } d = 1, 
\quad \ 
m_1 := 2 \ \ \text{if } d \geq 2
\end{equation}
and 
\begin{equation} \label{intro.def.Gamma}
\Gamma := \{ |k| : k \in \Z^d, \ k \neq 0 \} 
\subseteq \big\{ \sqrt{n} : n \in \N \big\}
\subset [1,\infty),
\end{equation}
where $|k| = (k_1^2 + \ldots + k_d^2)^{\frac12}$ is the usual Euclidean norm,
and $\N := \{ 1,2,\ldots \}$. 
Given a pair $(a,b)$ of functions, with 
\begin{equation} \label{intro.ab}
a(x) = \sum_{k \in \Z^d \setminus \{ 0 \}} a_k e^{ik \cdot x}, 
\quad \ 
b(x) = \sum_{k \in \Z^d \setminus \{ 0 \}} b_k e^{ik \cdot x}, 
\end{equation}
for each $\lm \in \Gamma$ we define 
\begin{equation} \label{intro.def.U}
U_\lm := U_\lm(a,b) := \sum_{|k| = \lm} ( \lm^3 |a_k|^2 + \lm |b_k|^2).
\end{equation}
We denote
\begin{align}
& \Gamma_0 := \Gamma_0(a,b) := \{ \lm \in \Gamma : U_\lm(a,b) = 0 \}, 
\notag \\ 
& \Gamma_1 := \Gamma_1(a,b) := \{ \lm \in \Gamma : U_\lm(a,b) > 0 \} 
= \Gamma \setminus \Gamma_0.
\label{intro.def.Gamma.1}
\end{align}

\begin{theorem} \label{thm:main}
There exist universal constants $\d \in (0,1)$, $C,A > 0$ with the following properties.
Let $\e, c_0$ be real numbers with 
\begin{equation} \label{intro.costantine}
0 < \e \leq \d c_0, \quad \ 
0 < c_0 \leq 1,
\end{equation}
and let 
\begin{equation} \label{intro.ab.small}
(a,b) \in H^{m_1 + \frac12}_0(\T^d,\R) \times H^{m_1 - \frac12}_0(\T^d,\R), 
\quad \ 
\| a \|_{m_1 + \frac12} + \| b \|_{m_1 - \frac12} \leq \e.
\end{equation}
Let $U_\lm = U_\lm(a,b)$, $\lm \in \Gamma$, be the sums defined in \eqref{intro.def.U}, 
and let $\Gamma_1 = \Gamma_1(a,b)$ be the set in \eqref{intro.def.Gamma.1}. 
Assume that $(a,b)$ satisfy 
\begin{equation} \label{intro.nonres}
| U_\a + U_\b - U_\lm | 
\geq c_0 (U_\a + U_\b + U_\lm)
\end{equation}
for all $\a,\b,\lm \in \Gamma_1$ such that $\a + \b = \lm$. 

Then the solution $(u,v)$ of system \eqref{p1} with initial conditions
$(u(0), v(0)) = (a,b)$ is defined on the time interval $[0,T]$, 
where 
\[
T = \frac{A c_0^3}{\e^6},
\]
with $(u,v) \in C([0,T],H^{m_1+\frac12}_0(\T^d,\R) \times H^{m_1-\frac12}_0(\T^d,\R))$
and 
\[
\| u(t) \|_{m_1+\frac12} + \| v(t) \|_{m_1 - \frac12} 
\leq C \e 
\quad \ \forall t \in [0,T].
\]
\end{theorem}

While assumptions \eqref{intro.costantine}, \eqref{intro.ab.small} are rather standard, 
assumption \eqref{intro.nonres} is specifically designed 
to avoid the triple resonances of the Kirchhoff equation, 
and it deserves some comments, 
which we collect in the next subsection.

\subsection{Nonresonance condition}
\label{sec:nonres}

In the following remarks we show that the set of functions satisfying 
the nonresonance condition \eqref{intro.nonres} is nonempty,
and in fact it contains several nontrivial examples;
we discuss here some aspects of that condition.

\begin{remark} \label{rem:scale.invariance}
(\emph{Invariance by constant factors}).
The nonresonance condition \eqref{intro.nonres} 
is invariant for multiplication by scalar constants: 
if $(a,b)$ satisfies \eqref{intro.nonres}, 
then, for all constants $\mu \in \R$, 
$(\mu a, \mu b)$ also satisfies \eqref{intro.nonres}
(with the same $c_0$). 

This means that the nonresonance condition 
and the smallness assumption in Theorem \ref{thm:main} 
are compatible: 
if a pair $(a,b)$ satisfies \eqref{intro.nonres} 
for some $c_0 > 0$, 
then $(\mu a, \mu b)$ sastisfies both \eqref{intro.nonres} 
(with the same $c_0$) and \eqref{intro.ab.small}
if $\mu$ is sufficiently small.  
\end{remark}

\begin{remark} \label{rem:decreasing}
(\emph{Decreasing sequences}).
Any decreasing sequence $(\s_\lm)_{\lm \in \Gamma}$ of nonnegative real numbers 
satisfies 
\[
|\s_\a + \s_\b - \s_\lm | \geq \frac13 (\s_\a + \s_\b + \s_\lm)
\]
for all $\a,\b,\lm \in \Gamma$ with $\a + \b = \lm$. 
To prove it, observe that 
$|\s_\a + \s_\b - \s_\lm | = \s_\a + \s_\b - \s_\lm$, 
and $\s_\lm \leq \min \{ \s_\a, \s_\b \} \leq \frac12(\s_\a + \s_\b)$ 
because $\lm > \a$, $\lm > \b$. 

As a consequence, any pair $(a,b)$ of functions 
such that $\lm \mapsto U_\lm(a,b)$ is decreasing 
satisfies \eqref{intro.nonres} with $c_0 = 1/3$.
\end{remark}

\begin{remark} \label{rem:power.decay}
(\emph{Fixed power decay}).
The observation of Remark \ref{rem:decreasing} applies, for example, 
to the sequence $\s_\lm = \lm^{-2\sigma}$, which is decreasing for $\s \geq 0$. 
Hence any pair $(a,b)$ of functions such that $U_\lm(a,b) = \lm^{-2\s}$ 
with $\s \geq 0$ satisfies \eqref{intro.nonres} with $c_0 = 1/3$.

The Sobolev regularity of such functions is the following. 
Since $\Gamma \subseteq \{ \sqrt{n} : n \in \N \}$ 
(where $\N := \{ 1,2,\ldots \}$),  
for any given $s \in \R$ we have 
\begin{align}
\| a \|_{s+\frac12}^2 + \| b \|_{s-\frac12}^2 
& = \sum_{k \in \Z^d} (|k|^{2s+1} |a_k|^2 + |k|^{2s-1} |b_k|^2)
\notag \\ & 
= \sum_{\lm \in \Gamma} \sum_{|k|=\lm} (\lm^{2s+1} |a_k|^2 + \lm^{2s-1} |b_k|^2)
\notag \\ & 
= \sum_{\lm \in \Gamma} \lm^{2s-2} U_\lm(a,b) 
= \sum_{\lm \in \Gamma} \frac{1}{\lm^{2\s-2s+2}} 
\leq \sum_{n \in \N} \frac{1}{n^{\s-s+1}},
\label{Sobolev} 
\end{align}
which is finite for $\s-s>0$. 
Thus $(a,b) \in H^{m_1 + \frac12}_0 \times H^{m_1 - \frac12}_0$ 
for $\s > m_1$. 
\end{remark}

\begin{remark} \label{rem:seq.choice}
(\emph{Sequential choice of $\s_\lm$}).
Let $0 < c_0 < 1$, and denote 
\[
\theta_1 := \frac{1-c_0}{1+c_0}\,, \quad \ 
\theta_2 := \frac{1+c_0}{1-c_0}\,.
\]
Let $(\s_\lm)_{\lm \in \Gamma}$ 
be a sequence of nonnegative real numbers, 
and let 
$\Gamma_1$ be the set of $\lm \in \Gamma$ 
such that $\s_\lm > 0$. 
The condition 
\begin{equation} \label{1106.2}
|\s_\a + \s_\b - \s_\lm| \geq c_0 (\s_\a + \s_\b + \s_\lm)
\quad \ \forall \a, \b, \lm \in \Gamma_1, \ \ 
\a + \b = \lm
\end{equation}
(which corresponds to \eqref{intro.nonres}) 
is equivalent to say that, 
for every $\lm \in \Gamma_1$, 
the number $\s_\lm$ does not belong 
to the finite union 
\[
G_\lm := \bigcup_{\begin{subarray}{c} \a, \b \in \Gamma_1 \\ \a + \b = \lm \end{subarray}}
I_{\a \b}
\]
of the open intervals 
\[
I_{\a \b} := \{ x \in \R : (\s_\a + \s_\b) \theta_1 < x < (\s_\a + \s_\b) \theta_2 \}. 
\]
For each $\lm$, $G_\lm$ is contained in an interval $(x_1,x_2)$ 
with $0 < x_1 \leq x_2 < \infty$, 
hence, once $\s_\a$ has been fixed for all $\a < \lm$, 
there are at least two intervals $[0,x_1]$ and $[x_2,\infty)$ 
where one can choose $\s_\lm$.

For example, fix $c_0 = \frac{1}{9}$. 
Then $\theta_1 = \frac{8}{10}$, 
$\theta_2 = \frac{10}{8}$.
Thus $\s_1$ has no restriction, 
$\s_2$ must be outside $I_{11} = (\frac{8}{10} 2 \s_1, \frac{10}{8} 2 \s_1)$; 
$\s_3$ must be outside $I_{12} = (\frac{8}{10} (\s_1 + \s_2) , \frac{10}{8} (\s_1 + \s_2))$; 
$\s_4$ must avoid $I_{22}$ and $I_{13}$, and so on; 
moreover, $\s_{\sqrt{2}}$ has no restriction, 
$\s_{\sqrt{8}}$ must be outside $I_{\sqrt{2} \sqrt{2}}$, etc.
For each integer $p$ that is the product of distinct prime numbers, 
there is no restriction on the choice of $\s_{\sqrt{p}}$.
\end{remark}

\begin{remark} \label{rem:odd}
(\emph{Absence of triplets: odd integers}). 
If the set $\Gamma_1$ does not contain any triplet $(\a,\b,\lm)$ with $\a+\b=\lm$, 
then \eqref{1106.2} is trivially satisfied. 
For example, this holds if $\Gamma_1 \subseteq \{ n \in \N : n \ \text{odd} \}$. 
Other examples can be constructed as lacunary subsets of $\N$.
\end{remark}

\begin{remark} \label{rem:primes}
(\emph{Arithmetic decomposition of $\Gamma$}).
The set $\Gamma$ can be decomposed as the disjoint union 
$\cup_p \, \Gamma(\sqrt{p})$
of the sets 
\[
\Gamma(\sqrt{p}) := \{ n \sqrt{p} : n \in \N \} \cap \Gamma,
\]
where $p$ is any product of distinct prime numbers. 

Hence, as a slightly more general version of the observation in Remark \ref{rem:odd},
\eqref{1106.2} is trivially satisfied if 
$\Gamma_1 \cap \Gamma(\sqrt{p}) \subseteq \{ n \sqrt{p} : n \ \text{odd} \}$
for all $p$.  

This decomposition of $\Gamma$ also implies that,
at least at the time scales we are concerned with in this paper, 
the ``effective system'' of homogeneity $\leq5$ (see \eqref{3005.1}-\eqref{3105.8}) that controls the time evolution of Sobolev norms for the Kirchoff equation in dimension $d \geq 2$ 
contains infinitely many copies 
of the same system in dimension $d=1$. 
These copies are {\em almost uncoupled}, since the only coupling comes from the factor $\cal P$ in \eqref{3105.8}, which is a function of time only and whose only effect is to produce a slight time rescaling.

In other words, the solutions of such an effective system
have essentially the same behavior 
in dimension 1 or higher. 
(The only thing that changes substantially with the dimension 
regards the regularity required by the normal forms, 
because denominators like $|k|-|j|$, $|k|+|j|-|\ell|$, $k,j,\ell \in \Z^d$, 
accumulates to zero if $d \geq 2$, while they are nonzero integers in dimension $d=1$;
see \cite{K old}, \cite{K mem} for more details).
\end{remark}

\begin{remark} \label{rem:perturb}
(\emph{Perturbations of \eqref{intro.nonres}}).
Given two pairs $(a,b)$, $(f,g)$ of functions, 
from the definition \eqref{intro.def.U} of $U_\lm$ one has 
\begin{equation*} 
U_\lm(a+f, b+g) = U_\lm(a,b) + U_\lm(f,g) + M_\lm(a,b,f,g)
\end{equation*}
where 
\[
M_\lm(a,b,f,g) = \sum_{|k|=\lm} 
\Big( \lm^3 (a_k \overline{f_k} + \overline{a_k} f_k)
+ \lm (b_k \overline{g_k} + \overline{b_k} g_k) \Big). 
\]
By Cauchy-Schwarz and H\"older's inequality, 
\begin{align*}
|M_\lm (a,b,f,g)|   
& \leq 2 \sum_{|k|=\lm} \Big( (\lm^{\frac32} |a_k|) (\lm^{\frac32} |f_k|) 
+ (\lm^{\frac12} |b_k|) (\lm^{\frac12} |g_k|) \Big)
\\ 
& \leq 2 \sum_{|k|=\lm} \Big( \lm^3 |a_k|^2 + \lm |b_k|^2 \Big)^{\frac12} 
\Big( \lm^3 |f_k|^2 + \lm |g_k|^2 \Big)^{\frac12}
\\ 
& \leq 2 \sqrt{U_\lm(a,b)} \, \sqrt{U_\lm(f,g)}.
\end{align*}
As a consequence, if $(f,g)$ satisfies 
\begin{equation} \label{U.mu.U}
U_\lm(f,g) \leq \mu^2 U_\lm(a,b)
\end{equation}
for some $\mu \geq 0$, 
then 
\begin{equation} \label{UUU}
\big| U_\lm (a+f, b+g) - U_\lm(a,b) \big| 
\leq (2 \mu + \mu^2) U_\lm(a,b).
\end{equation}
If, in addition, $(a,b)$ satisfies the nonresonance condition \eqref{intro.nonres}, 
then $(a+f, b+g)$ also satisfies \eqref{intro.nonres} with $c_0$ replaced by a smaller constant: 
by \eqref{UUU} we obtain
\begin{align*}
& | U_\a(a+f, b+g) + U_\b(a+f, b+g) - U_\lm(a+f, b+g) |
\\ 
& \quad \geq |U_\a(a,b) + U_\b(a,b) - U_\lm(a,b) |
- \sum_{\g = \a, \b, \lm} |U_\g(a+f,b+g) - U_\g(a,b)|
\\
& \quad \geq c_0 \sum_{\g=\a,\b,\lm} U_\g(a,b) 
- \sum_{\g = \a, \b, \lm} (2\mu + \mu^2) U_\g(a,b)
\\
& \quad \geq \frac{ c_0 - 2 \mu - \mu^2 }{ 1 + 2 \mu + \mu^2 } \, 
\sum_{\g = \a, \b, \lm} U_\g(a+f, b+g).
\end{align*}
We also note that 
\[
\frac{ c_0 - 2 \mu - \mu^2 }{ 1 + 2 \mu + \mu^2 } \, 
\geq c_0 - 4 \mu
\quad \ \forall \mu \geq 0, \ \ 0 < c_0 \leq 1.
\qedhere 
\]
\end{remark}

\begin{remark} \label{rem:fake ball}
(\emph{Translation of a ball in Sobolev norm}). 
We consider perturbations of the fixed decay example 
of Remark \ref{rem:power.decay}. 
Let $(a,b)$ be a pair of functions such that 
$U_\lm(a,b) = \lm^{-2\s}$ with $\s > m_1$. 
As observed in Remark \ref{rem:power.decay}, 
$(a,b)$ belongs to $H^{m_1+\frac12}_0 \times H^{m_1-\frac12}_0$
and satisfies \eqref{intro.nonres} with $c_0 = 1/3$.  
Let 
\[
(f,g) \in H^{s+\frac12}_0 \times H^{s - \frac12}_0, 
\quad \ 
\mu^2 := \| f \|_{s+\frac12}^2 + \| g \|_{s - \frac12}^2, 
\quad \ 
s := \s + 1.
\]
From the identities in \eqref{Sobolev}, one has 
\begin{equation} \label{ptw.sum}
\lm^{2s-2} U_\lm(f,g) \leq \sum_{\a \in \Gamma} \a^{2s-2} U_\a(f,g) 
= \| f \|_{s+\frac12}^2 + \| g \|_{s - \frac12}^2 = \mu^2,
\end{equation}
whence we deduce that 
\begin{equation} \label{ptw.bound}
U_\lm(f,g) \leq \frac{\mu^2}{\lm^{2s-2}} 
= \frac{\mu^2}{\lm^{2\s}} 
= \mu^2 U_\lm(a,b)
\quad \ \forall \lm \in \Gamma.
\end{equation}
Hence \eqref{U.mu.U} is verified, and, by Remark \ref{rem:perturb}, 
the pair $(a+f, b+g)$ satisfies the nonresonance condition \eqref{intro.nonres} 
with $c_0=\frac13 - 4 \mu$. 
If we take, for example, $\mu_0 := \frac{1}{24}$, 
then all pairs of functions in the set 
\[
\mB(a,b) := \big\{ (a,b) + (f,g) : \| f \|_{s+\frac12}^2 + \| g \|_{s-\frac12}^2 \leq \mu_0^2 = 1/576 \big\}
\]
satisfy the nonresonance condition \eqref{intro.nonres} 
with constant $c_0 = 1/6$. 

Note, however, that the set $\mB(a,b)$ is not a ball in the Sobolev space 
$H^{s+\frac12}_0 \times H^{s - \frac12}_0$, 
because $(a,b)$ does not belong to that space 
(since $s = \s+1$, the last series in \eqref{Sobolev} diverges). 
This ``gap of regularity'' 
is due to the fact 
that we have used the sum of the series in \eqref{ptw.sum}
to get the ``pointwise'' bound \eqref{ptw.bound} 
(namely a bound that holds at each single $\lm$).
\end{remark}

\begin{remark} \label{rem:other nonres}
(\emph{Other possible nonresonance conditions}). 
The nonresonance conditions \eqref{intro.nonres} 
can be replaced by other assumptions. 
Another possibility is to assume the ``Melnikov-like'' or ``Diophantine-like'' 
nonresonance conditions
\begin{equation} \label{alt.nonres}
|U_\a + U_\b - U_\lm| \geq \frac{c_0}{(\min \{ \a, \b, \lm \})^\tau}
\end{equation} 
for all $\a, \b, \lm \in \Gamma_1$ such that $\a + \b = \lm$, 
where $c_0, \tau$ are positive, fixed parameters.  
Inequalities similar as \eqref{alt.nonres} are perhaps more common in literature than \eqref{intro.nonres}. 
Both the fixed power decay example of Remark \ref{rem:power.decay} 
and its perturbations as in Remarks \ref{rem:perturb}, \ref{rem:fake ball} 
hold, after suitable adaptations, for the nonresonance conditions \eqref{alt.nonres}. 
A result very similar to Theorem \ref{thm:main} can be proved 
assuming \eqref{alt.nonres} instead of \eqref{intro.nonres}. 
The proof is also similar, just slightly more complicated. 
\end{remark}

\begin{remark} \label{rem:already small}
(\emph{Terms that are already small}). 
For any given $\e$, the nonresonance conditions \eqref{intro.nonres} 
need not be really satisfied by \emph{all} resonant triplets $\a + \b = \lm$ 
in $\Gamma_1$, because, using the decay of Fourier coefficients 
of functions in Sobolev spaces (like in \eqref{ptw.bound}), 
the terms $\int_0^T \th_{\a \b \lm}(t) \, dt$ 
that we estimate by integrating by parts in time
(see \eqref{3105.17}) are in fact already small 
if $\a, \b, \lm $ are sufficiently high 
(depending on $\e$). 
On the other hand, assuming that \eqref{intro.nonres} holds for all triplets $\a + \b = \lm$ 
(and not only for, say, $\a$ smaller than some power of $1/\e$) 
we directly obtain our result uniformly in $\e$.
\end{remark}

\subsection{Strategy of the proof}
\label{sec:strategy}

As already said, evolution PDEs on compact manifolds  
in general have no mechanism of global dispersion as time evolves. 
To obtain long-time existence for the solutions, 
an efficient strategy is to suitably tune the parameters of the equation,
avoiding their values corresponding to resonances; 
recent examples are the work \cite{Biasco Massetti Procesi} 
and, in the context of quasilinear PDEs, 
\cite{Berti Delort}, \cite{Berti Feola Pusateri}, \cite{Feola Iandoli}.

If the equation has no external parameter,  
to avoid the resonances one has, in general, nothing to tune except the initial data of the Cauchy problem; 
recent examples are \cite{BGHS}, \cite{Bernier Faou Grebert Rational}.
Since the Kirchhoff equation \eqref{K1} has no external parameter, 
we follow this approach, namely we select the initial data to avoid resonances. 

We start from the normal form of degree five, computed in the previous paper \cite{K mem}. 
The first remark is that the time evolution of the Sobolev norms of solutions of \eqref{K1} 
is fully described by the evolution of the ``superactions'' $S_\lm$ in \eqref{def SB}. 
Such an evolution, in turn, is governed by the ``effective system'' \eqref{3005.1}-\eqref{3105.8}. 
In particular, we focus on equation \eqref{3005.1}, which describes the evolution of the superactions. 
In this equation a crucial r\^{o}le is played by the complex factors $Z_{\a\b\lm}$ 
(see \eqref{def Z}). 
The basic idea is that, if the time derivative of the $Z_{\a\b\lm}$'s is bounded away from zero, one benefits from an ``\emph{averaging effect}'' that slows down the growth of the superactions. 
Since the time evolution of $Z_{\a\b\lm}$ is given by \eqref{3105.11}, this is the reason for introducing the nonresonance condition \eqref{hyp.nonres} on the initial datum. 
Such a nonresonance condition is stable under the normal form transformation (because of the very special structure of the Kirchhoff equation) and assumes the form \eqref{intro.nonres} in the original variables.  

In Proposition \ref{key} we prove the key ingredient: assuming that the initial data satisfy the nonresonance condition \eqref{hyp.nonres}, the aforementioned averaging effect allows to improve the \emph{a priori} bound for the evolution of the superactions $S_\lm$. The energy estimates based on the first step of normal form imply that the growth factor of the $S_\lm$, on a time interval of length $O(\e^{-4})$, is of order $O(1)$. Here, under the nonresonance condition \eqref{hyp.nonres}, we improve the bound on the growth factor of $S_\lm$ from $O(1)$ to $1+O(\e^2)$ (see \eqref{0106.9}). This improvement also guarantees that after a time of order $O(\e^{-4})$ the nonresonance condition is still satisfied (see \eqref{tesi.nonres}). Therefore, we are able to iterate the estimates of Lemma \ref{key} on a sequence of $O(\e^{-2})$ time intervals of length $O(\e^{-4})$. This is done in Lemma \ref{lemma:induz} and Lemma \ref{lemma:incolla}, and allows us to reach an existence time of order $O(\e^{-6})$.

\bigskip

\noindent
\textbf{Acknowledgements}.
This research is supported by the INdAM-GNAMPA Project 2019.

\section{Time evolution of the superactions}

\textbf{Notation.} 
In this paper ``$a \lesssim b$'' means 
``there exists a \emph{universal} constant $C>0$ such that $a \leq C b$''.
This notation is used in the proof of 
Lemma \ref{lemma:3005.3}, Lemma \ref{lemma:0106.1} 
and Proposition \ref{key}.

\medskip

For any real $s \geq 0$ we define the Sobolev space 
of pairs of complex conjugate functions
\begin{equation} \label{def H cc}
H^s_0(\T^d, c.c.) := \big\{ (u,v) \in H^s_0(\T^d,\C) \times H^s_0(\T^d,\C) : v = \overline{u} \big\}
\end{equation}
with norm $\| (u,v) \|_s := \| u \|_s = \| v \|_s$. 
The Fourier coefficients $u_k, v_k$ of $u,v$ satisfy 
$v_k = (\overline{u})_k = \overline{ u_{-k} }$.

In \cite{K old}-\cite{K mem} we proved that there exists 
a change of variable $\Phi = \Phi^{(1)} \circ \cdots \circ \Phi^{(5)}$
(a bounded quasilinear normal form transformation)
that transforms the Cauchy problem \eqref{K1}, \eqref{initial data 1} 
for the Kirchhoff equation into the problem
\[
\pa_t (u,v) = W(u,v), \quad \ 
(u,v)(0) = (u_0, v_0),
\]
which takes place in the spaces \eqref{def H cc}, 
where the vector field $W$ is in normal form except for a remainder $W_{\geq 7}$ 
of homogeneity order $\geq 7$ and for harmless terms that give zero contribution
to the energy estimate of the flow. 
The relevant formulas and estimates of the normal form construction 
are collected in the appendix, section \ref{sec:App A}.  

In \cite{K mem} we also introduced a simplified formulation of the equation
that puts together all the Fourier coefficients $u_k, v_k$ 
of frequencies $k$ on the same sphere $|k|=\lm$. 
The spheres in the Fourier space naturally appear 
--- spheres and not other geometrical objects ---  
because they are the set of all frequencies sharing the same eigenvalue of the Laplacian 
(and the Laplacian is the linear part of the vector field). The very special structure of the Kirchhoff equation allows us to write down an effective system (see \eqref{3005.1}-\eqref{3105.8}) involving only the global quantities $S_\lm, B_\lm$ (see \eqref{def SB} below) on each sphere. The evolution of such quantities (which governs the evolution of Sobolev norms) is independent from the (potentially much more complex) dynamics within each sphere.

In this section we consider the transformed equations on the Fourier spheres 
\eqref{3005.1}-\eqref{3105.8} as the starting point of our analysis; 
we refer to section \ref{sec:App A} for their derivation. 

Recall the definition \eqref{intro.def.Gamma} of $\Gamma$. 
For each $\lm \in \Gamma$, define 
\begin{equation} \label{def SB}
S_\lm := \sum_{k : |k| = \lm} |u_k|^2 
= \sum_{k : |k| = \lm} u_k v_{-k}, 
\quad \ 
B_\lm := \sum_{k : |k| = \lm} u_k u_{-k}. 
\end{equation}
Hence (remember that $v_{-k} = \overline{u_k}$) 
\begin{equation} \label{3105.2}
\overline{B_\lm} = \sum_{k : |k| = \lm} v_k v_{-k},
\quad \ 
\| u \|_s^2 = \sum_{\lm \in \Gamma} \lm^{2s} S_\lm. 
\end{equation}
Note that $S_\lm \geq 0$, $B_\lm \in \C$, and 
\begin{equation} \label{3105.1}
|B_\lm| \leq S_\lm
\end{equation}
(because $|u_k u_{-k}| \leq \frac12 (|u_k|^2 + |u_{-k}|^2)$). 
We call $S_\lm$ ``superactions''. 
 
By \eqref{3105.2}, the Sobolev norm $\| u \|_s$ 
is determined by the superactions $S_\lm$. 
To analyze the growth in time of each single $S_\lm$, 
we first observe a property of the vector field $W(u,v)$,
which is a consequence of the Fourier multiplier structure of the Kirchhoff equation 
and of the fact that all the transformations $\Phi^{(1)}, \ldots, \Phi^{(5)}$ 
preserve a similar structure on the transformed vector field.

\begin{lemma} \label{lemma:W geq 7 new}
There exist universal constants $\d>0$, $C>0$ such that
for all $(u,v) \in H^{m_1}_0(\T^d, c.c.)$ in the ball $\| u \|_{m_1} \leq \d$, 
for all $k \in \Z^d$, 
the $k$-th Fourier coefficient 
of the first component $(W_{\geq 7})_1(u,v)$ of $W_{\geq 7}(u,v)$ 
and the one of the second component $(W_{\geq 7})_2(u,v)$
both satisfy
\[
| [ (W_{\geq 7})_1(u,v) ]_k | \,, \, 
| [ (W_{\geq 7})_2(u,v) ]_k | \leq C \| u \|_{m_1}^6 (|u_k| + |u_{-k}|).
\]
\end{lemma}

\begin{proof} 
In the Appendix, section \ref{sec:App B}. 
\end{proof}

\subsection{Effective dynamics on Fourier spheres}

By \eqref{def SB}, \eqref{eq for uk}, \eqref{eq for vk}, 
for every $\lm \in \Gamma$ we calculate the equations 
for the evolution of $S_\lm, B_\lm$, which are
\begin{align}
\pa_t S_\lm 
& = \frac{3 i}{32}  \sum_{\begin{subarray}{c} \a,\b \in \Gamma \\  \a + \b = \lm \end{subarray}} 
( B_\a B_\b \overline{B_\lm}  
- \overline{B_\a} \overline{B_\b} B_\lm ) 
\a \b \lm
+ \frac{3 i}{16}  \sum_{\begin{subarray}{c} \a,\b \in \Gamma \\  \a - \b = \lm \end{subarray}}
( B_\a \overline{B_\b} \overline{B_\lm} 
- \overline{B_\a} B_\b B_\lm ) 
\a \b \lm 	
+ R_{S_\lm},
\label{3005.1}
\\ 
\pa_t B_\lm 
& = - 2 i (1 + \mP) \Big( \lm + \frac14 \lm^2 S_\lm \Big) B_\lm 
+ R_{B_\lm},
\label{3105.8}
\end{align}
where
\begin{align}
R_{S_\lm} 
& := \sum_{k : |k|=\lm} [ (W_{\geq 7})_1(u,v) ]_k v_{-k} 
+ \sum_{k : |k|=\lm} u_k [ (W_{\geq 7})_2(u,v) ]_{-k},
\label{3005.4} 
\\ 
R_{B_\lm}
& := \frac{i}{16} \sum_{\a \in \Gamma} |B_\a|^2 B_\lm \a^2 
\Big( \frac{1}{\a + \lm} - \frac{1 - \delta_\a^\lm}{\a - \lm} \Big)  
\notag \\ & \quad 
+ \frac{3i}{16} \sum_{\begin{subarray}{c} \a,\b \in \Gamma \\  \a + \b = \lm \end{subarray}} 
B_\a B_\b S_\lm \a \b \lm
+ \frac{i}{8} \sum_{\a \in \Gamma} S_\a S_\lm B_\lm \lm^2 \a 
\Big( 6 + \frac{\a}{\a + \lm} 
+ \frac{\a (1 - \delta_\a^\lm)}{\a - \lm} \Big) 
\notag \\ & \quad 
+ \frac{3i}{8} \sum_{\begin{subarray}{c} \a,\b \in \Gamma \\  \a - \b = \lm \end{subarray}} 
B_\a \overline{B_\b} S_\lm \a \b \lm
+ \sum_{k : |k|=\lm} 2 u_k [ (W_{\geq 7})_1(u,v) ]_{-k}.  
\label{3005.5}
\end{align}

\begin{lemma} \label{lemma:3005.3}
Let $(u,v) \in H^{m_1}_0(\T^d,c.c.)$ with $\| u \|_{m_1} \leq \d$, 
where $\d$ is the constant in Lemma \ref{lemma:W geq 7 new}. 
Then for all $\lm \in \Gamma$ the remainders 
defined in \eqref{3005.4}-\eqref{3005.5} satisfy
\begin{equation} \label{3005.3}
|R_{S_\lm}| \leq C \| u \|_{m_1}^6 S_\lm, \quad \ 
|R_{B_\lm}| \leq C \| u \|_{m_1}^4 S_\lm,
\end{equation}
where $C>0$ is a universal constant.
\end{lemma}

\begin{proof}
The estimate for $R_{S_\lm}$ follows from Lemma \ref{lemma:W geq 7 new}  
and the elementary inequality $(|u_k| + |u_{-k}|)^2 \leq 2 (|u_k|^2 + |u_{-k}|^2)$.
To estimate $R_{B_\lm}$, we note that 
\begin{equation} \label{3105.4}
\frac{1}{|\a-\lm|} \leq 3 \a \quad \ \forall \a, \lm \in \Gamma, \ \a \neq \lm
\end{equation}
in any dimension $d \geq 1$; 
for $d=1$ one has the stronger lower bound $|\a-\lm| \geq 1$ for $\a \neq \lm$.
Bound \eqref{3105.4} is not difficult to prove 
(see the proof of Lemma 4.1 in \cite{K old}). 
One also has the elementary inequality 
\begin{equation} \label{3105.5}
S_\a \a^{2p} 
\leq \sum_{\b \in \Gamma} S_\b \b^{2p} 
= \| u \|_p^2
\quad \ \forall \a \in \Gamma, \ \forall p \geq 0.
\end{equation}
Let $1^{st}, \ldots, 5^{th}$ denote the five sums in the r.h.s.\ of \eqref{3005.5}. 
By \eqref{3105.4}, \eqref{3105.1}, one has
\[ 
| 1^{st} | 
\lesssim \sum_\a |B_\a|^2 |B_\lm| \a^3 
\lesssim \sum_\a S_\a^2 \a^3 S_\lm
\lesssim \sum_\a S_\a \a \| u \|_1^2 S_\lm
\lesssim \| u \|_{\frac12}^2 \| u \|_1^2 S_\lm
\] 
because, by \eqref{3105.5}, $S_\a \a^2 \leq \| u \|_1^2$. 
Similarly, by \eqref{3105.4}, \eqref{3105.1}, 
\[ 
| 3^{rd} | 
\lesssim \sum_\a S_\a \a^3 S_\lm^2 \lm^2  
\lesssim \sum_\a S_\a \a^3 \| u \|_1^2 S_\lm
\lesssim \| u \|_{\frac32}^2 \| u \|_1^2 S_\lm
\] 
because, by \eqref{3105.5}, $S_\lm \lm^2 \leq \| u \|_1^2$.
In dimension $d=1$, 
using the lower bound $|\a-\lm| \geq 1$ instead of \eqref{3105.4},
one also has $|3^{rd}| \lesssim \| u \|_1^4 S_\lm$.
By \eqref{3105.4}, 
\[
|2^{nd}| 
\lesssim \sum_{\begin{subarray}{c} \a,\b \\  \a + \b = \lm \end{subarray}} 
S_\a S_\b S_\lm \a \b (\a+\b)
\lesssim \sum_{\a,\b} S_\a S_\b S_\lm \a \b (\a+\b)
\lesssim \| u \|_1^2 \| u \|_{\frac12}^2 S_\lm,
\]
and the same estimate also holds for $| 4^{th} |$ 
because $\lm = \a - \b \leq \a$. 
By Lemma \ref{lemma:W geq 7 new}, 
$|5^{th}| \lesssim \| u \|_{m_1}^6 S_\lm$.
Since $\| u \|_{m_1}^6 \leq \d^2 \| u \|_{m_1}^4$, 
the sum of the five terms gives the estimate for $R_{B_\lm}$.
\end{proof}

In Lemma \ref{lemma:0106.1} we observe that, 
in a time interval of order $\| u(0) \|_{m_1}^{-4}$, 
each single $S_\lm$ has a growth factor of at most order $1$.
First, we recall a result from \cite{K old}-\cite{K mem}.

\begin{lemma} \label{lemma:0106.3}
There exist universal positive constants 
$\d_1, C_1, A_1$ such that, for every initial data $(u_0, v_0) \in H^{m_1}_0(\T^d, c.c.)$
in the ball 
$\| u_0 \|_{m_1} \leq \d_1$, 
the Cauchy problem
\begin{equation} \label{0106.4}
\pa_t (u,v) = W(u,v), \quad \ 
(u,v)(0) = (u_0, v_0)
\end{equation}
has a unique solution 
$(u,v) \in C([0,T_1], H^{m_1}_0(\T^d, c.c.))$
on the time interval $[0,T_1]$, with 
\begin{equation} \label{0106.2}
\| u(t) \|_{m_1} \leq C_1 \| u_0 \|_{m_1} \leq \d 
\quad \ \forall t \in [0, T_1],
\end{equation}
where 
\begin{equation*} 
T_1 = A_1 \| u_0 \|_{m_1}^{-4}
\end{equation*}
and $\d$ is the constant in Lemma \ref{lemma:W geq 7 new}.
\end{lemma}

\begin{proof} 
Let $\Phi^{(1)}, \ldots, \Phi^{(4)}$ be the transformations in 
\eqref{1912.2}, 
\eqref{def Phi2}, 
\eqref{def Phi3}, 
\eqref{def Phi4}. 
In \cite{K old} (see ``Proof of Theorem 1.1'', just above the references in \cite{K old}) 
it is proved that the system 
\[
\pa_t (w,z) = X^+(w,z),
\]
namely the system obtained applying $\Phi^{(1)} \circ \cdots \circ \Phi^{(4)}$ 
to the original Kirchhoff equation, has local existence and uniqueness 
for initial data $(w_0, \overline{w_0}) \in H^{m_0}_0(\T^d, c.c.)$ 
in the ball $\| w_0 \|_{m_0} \leq \d_0$, 
and the solution $w(t)$ is well-defined on the time interval 
$[0, T_0]$, with $T_0 = A_0 \| w_0 \|_{m_0}^{-4}$.  
Moreover $\| w(t) \|_{m_0} \leq C_0 \| w_0 \|_{m_0}$ 
on $[0, T_0]$, and if, in addition, $w_0 \in H^s$ for some $s > m_0$, 
then $\| w(t) \|_s \leq C_0 \| w_0 \|_s$ on $[0, T_0]$;
$\d_0, A_0, C_0$ are universal constants.  

Then consider $\Phi^{(5)}$ in \eqref{def Phi5}. 
One has 
$\| \Phi^{(5)}(u,v) \|_{m_1} \leq (1 + C \| u \|_{m_1}^4) \| u \|_{m_1}$
for all $(u,v) \in H^{m_1}_0(\T^d,c.c.)$ 
(see Lemma \ref{lemma:ottobrata.5}),
and the inverse map $(\Phi^{(5)})^{-1}$ is well-defined on the ball $\| w \|_{m_1} \leq \d_0'$, 
with
\[
\| (\Phi^{(5)})^{-1}(w,z) \|_{m_1} \leq 2 \| w \|_{m_1}
\]
for all $\| w \|_{m_1} \leq \d_0'$ (see Lemma \ref{lemma:Phi5 inv}); 
$C, \d_0'$ are universal constants.  
As a consequence, the system $\pa_t (u,v) = W(u,v)$ in \eqref{3101.16},
namely the system obtained applying the transformation $(w,z) = \Phi^{(5)}(u,v)$, 
satisfies the property of the statement, taking $\d_1$ sufficiently small.
\end{proof}

\begin{lemma} \label{lemma:0106.1}
Let $(u_0, v_0) \in H^{m_1}_0(\T^d, c.c.)$, 
$\| u_0 \|_{m_1} \leq \d_1$, with $\d_1$ given in Lemma \ref{lemma:0106.3}. 
Let $(u(t), v(t))$ be the solution of the Cauchy problem \eqref{0106.4}
on the time interval $[0, T_1]$, $T_1 = A_1 \| u_0 \|_{m_1}^{-4}$,
given in Lemma \ref{lemma:0106.3}. 
For every $t \in [0,T_1]$, let $S_\lm(t)$ be the sum defined by \eqref{def SB}.
Then 
\begin{align} 
\label{0106.15}
|\pa_t S_\lm(t)| 
& \leq C \| u_0 \|_{m_1}^4 S_\lm(t),
\\ 
C' S_\lm(0) 
\leq S_\lm(0) e^{- C \| u_0 \|_{m_1}^4 t}
& \leq S_\lm(t) 
\leq S_\lm(0) e^{C \| u_0 \|_{m_1}^4 t} 
\leq C'' S_\lm(0),
\label{0106.11}
\end{align}
for all $t \in [0, T_1]$, 
for all $\lm \in \Gamma$, 
where $C, C', C''$ are universal constants. 
\end{lemma}

\begin{proof}
Since $(u,v)$ solves \eqref{0106.4}, 
$S_\lm$ satisfies equation \eqref{3005.1} for all $t \in [0, T_1]$. 
Moreover, by \eqref{0106.2}, $u(t)$ remains in the ball $\| u \|_{m_1} \leq \d$ 
on the time interval $[0, T_1]$, and therefore the estimates of previous lemmas apply. 
Then, with estimates similar to those in the proof of Lemma \ref{lemma:3005.3}, one has 
\begin{align*}
\Big| \sum_{\begin{subarray}{c} \a,\b \in \Gamma \\  \a + \b = \lm \end{subarray}} 
( B_\a B_\b \overline{B_\lm} - \overline{B_\a} \overline{B_\b} B_\lm ) \a \b \lm \Big| 
& \lesssim \sum_{\a,\b \in \Gamma} S_\a S_\b S_\lm \a \b (\a + \b) 
\lesssim \| u \|_1^2 \| u \|_{\frac12}^2 S_\lm,
\notag \\ 
\Big| \sum_{\begin{subarray}{c} \a,\b \in \Gamma \\  \a - \b = \lm \end{subarray}}
( B_\a \overline{B_\b} \overline{B_\lm} - \overline{B_\a} B_\b B_\lm ) \a \b \lm \Big|
& \lesssim \sum_{\a,\b \in \Gamma} S_\a S_\b S_\lm \a^2 \b 
\lesssim \| u \|_1^2 \| u \|_{\frac12}^2 S_\lm,
\notag \\ 
|R_{S_\lm}| 
& \lesssim \| u \|_{m_1}^6 S_\lm.
\end{align*}
Hence, by \eqref{3005.1}, 
\[
|\pa_t S_\lm(t)| \leq C \| u(t) \|_{m_1}^4 S_\lm(t) 
\quad \ \forall t \in [0, T_1],
\]
and therefore, by \eqref{0106.2}, we obtain \eqref{0106.15}. 
Then \eqref{0106.11} follows by Gronwall's inequality. 
\end{proof}

Our goal is to improve the growth factor of $S_\lm$, 
whose evolution is driven by equation \eqref{3005.1}. 
Thus we analyze the terms in \eqref{3005.1}. 
Denote
\begin{equation}\label{def Z}
Z_{\a\b\lm} := B_\a B_\b \overline{B_\lm}.
\end{equation}
Using \eqref{3105.8}, we calculate
\begin{align} 
\pa_t Z_{\a\b\lm}
& = -2i (1 + \mP) \Big(\a + \b - \lm +\frac14 (\a^2 S_\a + \b^2 S_\b - \lm^2 S_\lm)\Big) Z_{\a\b\lm}
+ \widetilde R_{Z_{\a\b\lm}}
\label{3105.9}
\end{align}
where 
\begin{align} 
\widetilde R_{Z_{\a\b\lm}}
& := 
R_{B_\a} B_\b \overline{B_\lm} 
+ B_\a R_{B_\b} \overline{B_\lm} 
+ B_\a B_\b \overline{R_{B_\lm}}.
\label{3105.10}
\end{align}
For $\a + \b = \lm$,
isolating the first nontrivial contribution from terms of higher homogeneity orders, 
one has 
\begin{equation}  \label{3105.11}
\pa_t Z_{\a\b\lm}
= - \frac{i}{2} (\a^2 S_\a + \b^2 S_\b - \lm^2 S_\lm) Z_{\a\b\lm} 
+ R_{Z_{\a\b\lm}}
\end{equation}
where 
\begin{equation}  \label{3105.12}
R_{Z_{\a\b\lm}}
:= - \frac{i}{2} \mP (\a^2 S_\a + \b^2 S_\b - \lm^2 S_\lm) Z_{\a\b\lm} 
+ \widetilde R_{Z_{\a\b\lm}}.
\end{equation}

\begin{lemma} \label{lemma:3105.13}
Let $(u,v) \in H^{m_1}_0(\T^d,c.c.)$ with $\| u \|_{m_1} \leq \d$, 
where $\d$ is the constant in Lemma \ref{lemma:W geq 7 new}. 
Then for all $\a,\b,\lm \in \Gamma$ with $\a+\b=\lm$
the remainder defined in \eqref{3105.12} satisfies 
\begin{equation} \label{0106.10}
|R_{Z_{\a\b\lm}}| \leq C \| u \|_{m_1}^4 S_\a S_\b S_\lm,
\end{equation}
where $C>0$ is a universal constant.
\end{lemma}

\begin{proof}
By \eqref{3105.1}, \eqref{lemma:3005.3}, one has immediately 
$| \widetilde R_{Z_{\a\b\lm}} | \leq C \| u \|_{m_1}^4 S_\a S_\b S_\lm$. 
By \eqref{stima mP}, \eqref{stima mM}, for $\| u \|_{m_1} \leq \d$ one has 
\[
0 \leq \mP(\Phi^{(5)}(u,v)) 
\leq C \| \Phi^{(5)}(u,v) \|_{\frac12}^2
\leq C \| u \|_{\frac12}^2.
\]
Moreover 
\begin{align}
& | \a^2 S_\a + \b^2 S_\b - \lm^2 S_\lm | 
\leq 3\sum_{\g \in \Gamma} \g^2 S_\g 
= 3\| u \|_1^2,
\qquad
 |Z_{\a\b\lm}| 
\leq S_\a S_\b S_\lm,
\label{3105.19}
\end{align}
and the first term in the r.h.s.\ of \eqref{3105.12} 
is also bounded by $C \| u \|_{m_1}^4 S_\a S_\b S_\lm$. 
\end{proof}

In the next proposition we prove that,
if the term $(\a^2 S_\a + \b^2 S_\b - \lm^2 S_\lm)$ in \eqref{3105.11} 
is bounded away from zero 
with a quantitative lower bound (``nonresonance condition''),
the growth factor of each single $S_\lm$ is smaller than its a priori estimate.

\begin{proposition} \label{key}
There exist universal constants $A_*, K_* > 0$ with the following properties.
Let $0 < c_0 \leq 1$. 
Let $\rag > 0$, $(u_0, v_0) \in H^{m_1}_0(\T^d,c.c.)$, with 
\begin{equation} \label{1106.8}
\| u_0 \|_{m_1} \leq \rag \leq \d_1, 
\end{equation}
where $\d_1$ is given in Lemma \ref{lemma:0106.3}. 
Let $(u,v)$ be the solution of the Cauchy problem \eqref{0106.4} 
on the interval $[0, T_1]$ given by Lemma \ref{lemma:0106.3}, 
and let $S_\lm(t)$ be its superactions at time $t$. Let 
\begin{equation*} 
\Gamma_0 := \{ \lm \in \Gamma : S_\lm(0) = 0 \}, 
\quad \ 
\Gamma_1 := \{ \lm \in \Gamma : S_\lm(0) > 0 \} = \Gamma \setminus \Gamma_0.
\end{equation*}
Assume that, at time $t=0$, the datum $u_0$ satisfies the 
``nonresonance condition'' 
\begin{equation} \label{hyp.nonres}
\begin{aligned}
& \big| \a^2 S_\a(0) + \b^2 S_\b(0) - \lm^2 S_\lm(0) \big| 
\geq c_0 \big( \a^2 S_\a(0) + \b^2 S_\b(0) + \lm^2 S_\lm(0) \big)
\\ 
& \, \forall \a,\b,\lm \in \Gamma_1, 
\ \ \a + \b = \lm.
\end{aligned}
\end{equation}
Let 
\begin{equation} \label{def.T2}
T_* := A_* \rag^{-4} c_0.
\end{equation}
One has $T_* \leq T_1$, 
\begin{align} \label{0106.9}
|S_\lm(t) - S_\lm(0)| & \leq K_* c_0^{-2} \rag^2 S_\lm(0) 
\quad \ \forall t \in [0,T_*], 
\ \ \forall \lm \in \Gamma,
\\ 
\label{1109.9}
\| u(t) \|_{m_1} 
& \leq (1 + K_* c_0^{-2} \rag^2) \| u_0 \|_{m_1}
\quad \ \forall t \in [0,T_*], 
\end{align}
and 
\begin{equation} \label{tesi.nonres}
\begin{aligned}
& \big| \a^2 S_\a(t) + \b^2 S_\b(t) - \lm^2 S_\lm(t) \big| 
\geq c_1 \big( \a^2 S_\a(t) + \b^2 S_\b(t) + \lm^2 S_\lm(t) \big)
\\ 
& \, \forall t \in [0,T_*], \quad 
\, \forall \a,\b,\lm \in \Gamma_1, 
\ \ \a + \b = \lm,
\end{aligned}
\end{equation}
where 
\begin{equation} \label{1106.4}
c_1 = c_0 (1 - K_* c_0^{-2} \rag^2).
\end{equation}
\end{proposition}

\begin{proof}
%
For $\a,\b,\lm \in \Gamma$ with $\a + \b = \lm$ we denote 
\begin{equation} \label{def om}
\begin{aligned} 
\om_{\a \b \lm} & := \a^2 S_\a + \b^2 S_\b - \lm^2 S_\lm, 
\\ 
\Om_{\a \b \lm} & := \a^2 S_\a + \b^2 S_\b + \lm^2 S_\lm.
\end{aligned}
\end{equation}
Since $(u,v)$ solves \eqref{0106.4} on $[0, T_1]$, 
all $S_\lm(t)$, and therefore all $\om_{\a\b\lm}(t)$, $\Om_{\a\b\lm}(t)$, 
are defined for $t \in [0,T_1]$. 
From \eqref{0106.11}, for all $t \in [0,T_1]$ one has 
\begin{equation} \label{0906.9}
S_\lm(t) > 0 \ \text{ if} \ \lm \in \Gamma_1, 
\quad 
S_\lm(t) = 0 \ \text{ if} \ \lm \in \Gamma_0
\end{equation}
(the Fourier support is invariant for the Kirchhoff equation).
For $\a, \b, \lm \in \Gamma_1$ with $\a+\b=\lm$, 
by assumption \eqref{hyp.nonres} one has 
\begin{equation} \label{0906.4}
|\om_{\a\b\lm}(0)| \geq c_0 \Om_{\a\b\lm}(0).
\end{equation}
Using \eqref{0106.15}, \eqref{1106.8}, \eqref{0106.11}, 
for all $t \in [0,T_1]$ one has 
\begin{align}
|\pa_t \om_{\a \b \lm}(t)| 
& = |\a^2 \pa_t S_\a(t) + \b^2 \pa_t S_\b(t) - \lm^2 \pa_t S_\lm(t)|
\notag \\ & \leq \a^2 |\pa_t S_\a(t)| + \b^2 |\pa_t S_\b(t)| + \lm^2 |\pa_t S_\lm(t)|
\notag \\ & \leq C \| u_0 \|_{m_1}^4 (\a^2 S_\a(t) + \b^2 S_\b(t) + \lm^2 S_\lm(t))
\leq C \rag^4 \Om_{\a\b\lm}(t)
\leq A \rag^4 \Om_{\a\b\lm}(0)
\label{0906.10}
\end{align}
where $A>0$ is a universal constant. 
Thus 
\begin{equation} \label{0906.5}
|\om_{\a\b\lm}(t) - \om_{\a\b\lm}(0)|
\leq \int_0^t |\pa_t \om_{\a\b\lm}(s)| \, ds
\leq A \rag^4 t \,\Om_{\a\b\lm}(0).
\end{equation}
By \eqref{0906.4}, \eqref{0906.5} one has
\begin{align*} 
|\om_{\a\b\lm}(t)| 
& \geq |\om_{\a\b\lm}(0)| - |\om_{\a\b\lm}(t) - \om_{\a\b\lm}(0)|
\geq \big( c_0 - A \rag^4 t \big) \, \Om_{\a\b\lm}(0).
\end{align*}
Therefore 
\begin{equation} \label{time.nonres}
|\om_{\a\b\lm}(t)| \geq \frac{c_0}{2} \, \Om_{\a\b\lm}(0) 
\quad \ \forall t \in [0,T_*], 
\ \  \a, \b, \lm \in \Gamma_1, 
\ \ \a + \b = \lm,
\end{equation}
where
\begin{equation} \label{0906.7}
T_* := A_* c_0 \rag^{-4}, 
\quad \ 
A_* := \min \Big\{ A_1, \frac{1}{2A} \Big\} 
\end{equation}
and $A_1$ is given by Lemma \ref{lemma:0106.3},
so that
\[
T_* \leq T_1 = A_1 \| u_0 \|_{m_1}^{-4},
\quad \ 
(c_0 - A \rag^4 T_*) \geq \tfrac12 c_0.
\]

Recalling \eqref{3005.1},
to analyze the difference $S_\lm(T) - S_\lm(0)$ 
we study the integral of the imaginary part of $Z_{\a\b\lm}$ on $[0,T]$, 
for any $T \in [0,T_*]$, and $\a,\b,\lm \in \Gamma$ with $\a + \b = \lm$.
Let
\begin{equation}\label{def th}
\nth_{\a\b\lm} := {\rm Im} (Z_{\a\b\lm}) = \frac{B_\a B_\b \overline{B_\lm} - \overline{B_\a B_\b}B_\lm}{2i}.
\end{equation}
If at least one among $\a,\b,\lm$ belongs to $\Gamma_0$, 
say $\a \in \Gamma_0$, then, as observed in \eqref{0906.9}, 
the corresponding $S_\a(t)$ is zero for all $t \in [0,T_*]$. 
Hence, by \eqref{3105.1}, 
$B_\a(t)$ is also identically zero on $[0,T_*]$, 
and therefore $\nth_{\a\b\lm}(t) = 0$ 
by its definition \eqref{def th}. 
Thus, for all $T \in [0, T_*]$, 
\begin{equation} \label{0906.8}
\int_0^T \nth_{\a\b\lm}(t) \, dt = 0
\quad \ \text{if }  
\{ \a, \b, \lm \} \cap \Gamma_0 \neq \emptyset.
\end{equation}

It remains to study the case in which $\a, \b, \lm$ all belong to $\Gamma_1$.
Since $(u,v)$ solves \eqref{0106.4} on $[0,T_*] \subseteq [0, T_1]$, 
$Z_{\a\b\lm}$ solves \eqref{3105.11} on the same time interval, 
namely
\[
\pa_t Z_{\a\b\lm}(t) 
= - \frac{i}{2} \, \om_{\a \b \lm}(t) Z_{\a\b\lm}(t) 
+ R_{Z_{\a\b\lm}}(t)
\quad \ \forall t \in [0, T_*].
\]
By the lower bound \eqref{time.nonres}, 
$\om_{\a \b \lm}(t)$ is a nonzero real number, 
therefore we can divide by it and we obtain
\begin{equation} \label{3105.14}
Z_{\a\b\lm} (t)
= \frac{2 i}{\om_{\a \b \lm}(t)} \, 
\big( \pa_t Z_{\a\b\lm}(t) - R_{Z_{\a\b\lm}}(t) \big)
\quad \ \forall t \in [0, T_*].
\end{equation}
By \eqref{3105.14}, and integrating by parts, for all $T \in [0,T_*]$ we have
\begin{align} 
\int_0^T Z_{\a\b\lm}(t) \, dt 
& = 2 i \int_0^T \frac{\pa_t Z_{\a \b \lm}(t)}{\om_{\a \b \lm}(t)} \, dt 
- 2i \int_0^T \frac{R_{Z_{\a \b \lm}}(t)}{\om_{\a \b \lm}(t)} \, dt
\notag \\ 
& = 2i \frac{Z_{\a \b \lm}(T)}{\om_{\a \b \lm}(T)}
- 2i \frac{Z_{\a \b \lm}(0)}{\om_{\a \b \lm}(0)}
+ 2i \int_0^T Z_{\a \b \lm}(t) \frac{\pa_t \om_{\a \b \lm}(t)}{(\om_{\a \b \lm}(t))^2} \, dt
- 2i \int_0^T \frac{R_{Z_{\a \b \lm}}(t)}{\om_{\a \b \lm}(t)} \, dt.
\label{3105.17}
\end{align}
By \eqref{3105.19}, \eqref{time.nonres}, one has 
\begin{equation} \label{0106.12}
\frac{|Z_{\a \b \lm}(0)|}{|\om_{\a \b \lm}(0)|} 
\lesssim \frac{ S_\a(0) S_\b(0) S_\lm(0)}{ c_0 \Om_{\a\b\lm}(0)}.
\end{equation}
By \eqref{3105.19}, \eqref{time.nonres}, then \eqref{0106.11}, 
one also has 
\begin{equation} \label{0106.13}
\frac{|Z_{\a \b \lm}(T)|}{|\om_{\a \b \lm}(T)|} 
\lesssim \frac{S_\a(T) S_\b(T) S_\lm(T)}{c_0 \Om_{\a \b \lm}(0)} 
\lesssim \frac{S_\a(0) S_\b(0) S_\lm(0)}{c_0 \Om_{\a \b \lm}(0)}. 
\end{equation}
By \eqref{0106.10}, \eqref{time.nonres}, 
then \eqref{0106.11}, \eqref{0106.2}, 
and then \eqref{0906.7},  
one has 
\begin{align}
\Big| \int_0^T \frac{R_{Z_{\a \b \lm}}(t)}{\om_{\a \b \lm}(t)} \, dt \Big| 
& \lesssim \int_0^T \frac{\| u(t) \|_{m_1}^4 S_\a(t) S_\b(t) S_\lm(t)}{c_0 \Om_{\a\b\lm}(0)}\, dt
\notag 
\\
& \lesssim \int_0^T \frac{\| u_0 \|_{m_1}^4 S_\a(0) S_\b(0) S_\lm(0)}{c_0 \Om_{\a\b\lm}(0)}\, dt
\lesssim 
\frac{S_\a(0) S_\b(0) S_\lm(0)}{\Om_{\a\b\lm}(0)}.
\label{0106.14}
\end{align}
By \eqref{3105.19},
\eqref{0106.11},
\eqref{0906.10},
\eqref{time.nonres},
\eqref{0906.7}, 
we obtain
\begin{align}
\Big| \int_0^T Z_{\a \b \lm}(t)  
\frac{\pa_t \om_{\a \b \lm}(t)}{(\om_{\a \b \lm}(t))^2} \, dt \Big| 
\lesssim \int_0^T \frac{S_\a(0) S_\b(0) S_\lm(0) \| u_0 \|_{m_1}^4 \Om_{\a\b\lm}(0)}
{ [ c_0 \Om_{\a\b\lm}(0) ]^2 } \, dt
\lesssim \frac{ S_\a(0) S_\b(0) S_\lm(0)}{c_0 \Om_{\a\b\lm}(0)}.
\label{0106.17}
\end{align}
From formula \eqref{3105.17} 
and estimates \eqref{0106.12}, 
\eqref{0106.13}, 
\eqref{0106.14}, 
\eqref{0106.17} 
we deduce that 
\begin{align}
\Big| \int_0^T Z_{\a \b \lm}(t) \, dt \Big| 
& \lesssim 
\frac{ S_\a(0) S_\b(0) S_\lm(0)}{ c_0 \Om_{\a\b\lm}(0) }
\quad \ \forall \a,\b,\lm \in \Gamma_1, 
\ \ \a + \b = \lm. 
\label{0106.18}
\end{align}

Now we fix any $\lm \in \Gamma_1$ 
and write equation \eqref{3005.1} in terms of $\nth_{\a \b \lm}$ (defined in \eqref{def th}), 
namely 
\begin{equation} \label{0906.11}
\pa_t S_\lm(t) 
= - \frac{3}{16} 
\sum_{\begin{subarray}{c} \a,\b \in \Gamma \\  \a + \b = \lm \end{subarray}} 
\nth_{\a \b \lm}(t) \, \a \b \lm 
+ \frac{3}{8} 
\sum_{\begin{subarray}{c} \a,\b \in \Gamma \\  \b + \lm = \a \end{subarray}} 
\nth_{\b \lm \a}(t) \, \a \b \lm
+ R_{S_\lm}(t).
\end{equation}
The first sum in \eqref{0906.11} has a finite number of terms;
the second sum is a series of functions that converges uniformly on $[0,T_*]$
because, by \eqref{3105.19}, \eqref{0106.11}, 
\begin{align*}
\sum_{\begin{subarray}{c} \a,\b \in \Gamma \\  \b + \lm = \a \end{subarray}}
\big( \sup_{t \in [0,T_*]} |\nth_{\b \lm \a}(t)| \big) \, \a \b \lm
& \lesssim \sum_{\begin{subarray}{c} \a,\b \in \Gamma \\  \b + \lm = \a \end{subarray}}
S_\a(0) S_\b(0) S_\lm(0) \a \b \lm
\\ & 
\lesssim \Big( \sum_{\a \in \Gamma} S_\a(0) \a \Big) 
\Big( \sum_{\b \in \Gamma} S_\b(0) \b \Big) 
\lm S_\lm(0) 
\lesssim \| u_0 \|_{\frac12}^4 \lm S_\lm(0) < \infty.
\end{align*}
Therefore the sums in \eqref{0906.11} can be integrated term by term. 
For all $T \in [0, T_*]$, integrating \eqref{0906.11} on $[0,T]$ gives
\begin{align} 
S_\lm(T) - S_\lm(0) 
= -\frac{3}{16} 
\sum_{\begin{subarray}{c} \a,\b \in \Gamma \\  \a + \b = \lm \end{subarray}} 
\int_0^T \nth_{\a \b \lm} \, dt \, \a \b \lm 
+ \frac38 
\sum_{\begin{subarray}{c} \a,\b \in \Gamma \\  \b + \lm = \a \end{subarray}} 
\int_0^T \nth_{\b \lm \a} \, dt \, \b \lm \a
+ \int_0^T R_{S_\lm}(t) \, dt.
\label{0206.1}
\end{align}
Using \eqref{3005.3}, 
\eqref{0106.2}, \eqref{0106.11},  
\eqref{0906.7},
one has 
\begin{align}
\label{0306.8}
\Big| \int_0^T R_{S_\lm}(t) \, dt \Big| 
\lesssim \int_0^T \| u_0 \|_{m_1}^6 S_\lm(0) \, dt
\lesssim \rag^2 S_\lm(0).
\end{align}
Since $\lm \in \Gamma_1$, 
we use \eqref{0906.8} for the terms in \eqref{0206.1} 
with $\a$, or $\b$, or both $\a,\b \in \Gamma_0$, 
we use \eqref{0106.18} for the terms with both $\a,\b \in \Gamma_1$,
and we obtain  
\begin{align*} 
|S_\lm(T) - S_\lm(0)| 
\lesssim 
\sum_{\begin{subarray}{c} \a,\b \in \Gamma_1 \\  \a + \b = \lm \end{subarray}} 
\frac{S_\a(0) S_\b(0) S_\lm(0) \a \b \lm}{c_0 \Om_{\a\b\lm}(0)} 
+ 
\sum_{\begin{subarray}{c} \a,\b \in \Gamma_1 \\  \b + \lm = \a \end{subarray}} 
\frac{S_\a(0) S_\b(0) S_\lm(0) \a \b \lm}{c_0 \Om_{\a\b\lm}(0)} 
+ \rag^2 S_\lm(0).
\end{align*}
For $\a+\b=\lm$ one has $\a \b \lm = \a^2 \b + \a \b^2$, 
and  
\begin{align*}
\sum_{\begin{subarray}{c} \a,\b \in \Gamma_1 \\  \a + \b = \lm \end{subarray}} 
\frac{S_\a(0) S_\b(0) \a \b \lm}{\Om_{\a\b\lm}(0)} \, \frac{S_\lm(0)}{c_0}
& 
= \Big( \sum_{\begin{subarray}{c} \a,\b \in \Gamma_1 \\  \a + \b = \lm \end{subarray}} 
\frac{\a^2 S_\a(0)}{\Om_{\a\b\lm}(0)} \,  \b S_\b(0) 
+ \sum_{\begin{subarray}{c} \a,\b \in \Gamma_1 \\  \a + \b = \lm \end{subarray}} 
\frac{\b^2 S_\b(0)}{\Om_{\a\b\lm}(0)} \, \a S_\a(0) \Big) 
c_0^{-1} S_\lm(0)
\\
& \leq \Big( \sum_{\b \in \Gamma} \b S_\b(0) + \sum_{\a \in \Gamma} \a S_\a(0) \Big) c_0^{-1} S_\lm(0)
= 2 \| u_0 \|_{\frac12}^2 c_0^{-1} S_\lm(0)
\end{align*}
by the definition \eqref{def om} of $\Om_{\a\b\lm}$. 
For $\b+\lm=\a$ one has $\a \b \lm = \a^2 \b - \a \b^2 \leq \a^2 \b$, 
and the sum is estimated similarly. 
Hence 
\begin{equation} 
|S_\lm(T) - S_\lm(0)| 
\leq C_* \rag^2 c_0^{-1} S_\lm(0)
\quad \ \forall T \in [0,T_*] 
\label{0906.12}
\end{equation}
for all $\lm \in \Gamma_1$, 
for some universal constant $C_*>0$.
For $\lm \in \Gamma_0$ one has \eqref{0906.9}, 
therefore \eqref{0906.12} holds for all $\lm \in \Gamma$. 

From \eqref{0906.9} 
we deduce that
\begin{align*}
\| u(t) \|_{m_1}^2 
= \sum_{\lm \in \Gamma} S_\lm(t) \lm^{2m_1} 
& \leq \sum_{\lm \in \Gamma} (1 + C_* \rag^2 c_0^{-1}) S_\lm(0) \lm^{2m_1}  
= (1 + C_* \rag^2 c_0^{-1}) \| u_0 \|_{m_1}^2.
\end{align*}
Taking the square root 
and using the elementary bound $\sqrt{1+x} \leq 1+x$ 
(which holds for all $x \geq 0$) 
we obtain 
\begin{equation} \label{1206.1}
\| u(t) \|_{m_1} 
\leq (1 + C_* \rag^2 c_0^{-1}) \| u_0 \|_{m_1}
\quad \ \forall t \in [0,T_*].
\end{equation}
 
By triangular inequality, from \eqref{0906.12} it follows that 
\begin{align} 
| \om_{\a \b \lm}(t) - \om_{\a \b \lm}(0)| 
& \leq \a^2 |S_\a(t) - S_\a(0)| 
+ \b^2 |S_\b(t) - S_\b(0)| 
+ \lm^2 |S_\lm(t) - S_\lm(0)| 
\notag \\ & 
\leq C_* \rag^2 c_0^{-1} \Om_{\a \b \lm}(0)
\qquad \forall t \in [0,T_*], \ \a,\b,\lm \in \Gamma.
\label{0906.13}
\end{align}
For $\a, \b, \lm \in \Gamma_1$, $\a + \b = \lm$, 
by \eqref{0906.4}, \eqref{0906.13} we obtain
\begin{align}
|\om_{\a \b \lm}(t)| 
& \geq |\om_{\a \b \lm}(0)| 
- |\om_{\a \b \lm}(t) - \om_{\a \b \lm}(0)|  
\geq \big( c_0 - C_* \rag^2 c_0^{-1} \big) \Om_{\a\b\lm}(0)
\label{1106.6}
\end{align}
for all $t \in [0, T_*]$. 
By \eqref{0906.12},
\begin{equation} \label{1106.7}
\Om_{\a\b\lm}(t) \leq (1 + C_* \rag^2 c_0^{-1} \big) \Om_{\a\b\lm}(0)
\quad \ \forall t \in [0,T_*].
\end{equation}
From \eqref{1106.6}, \eqref{1106.7} we get
\[
|\om_{\a \b \lm}(t)| 
\geq \frac{ c_0 - C_* \rag^2 c_0^{-1} }{ 1 + C_* \rag^2 c_0^{-1} } \, \Om_{\a\b\lm}(t).
\]
The elementary inequality 
\[
\frac{c_0 - x}{1+x} \geq c_0 -2x
\]
holds for all $x \geq 0$ because $c_0 \leq 1$; 
we apply it with $x = C_* \rag^2 c_0^{-1}$, 
and obtain \eqref{tesi.nonres}-\eqref{1106.4} 
with $K_* := 2 C_*$. 
Finally we deteriorate \eqref{0906.12}, \eqref{1206.1} 
replacing $C_*$ with $K_*$ 
and $c_0^{-1}$ with $c_0^{-2}$, 
and we obtain \eqref{0106.9}, \eqref{1109.9}.
\end{proof}

Now we apply repeatedly Proposition \ref{key} 
and use the improved growth estimate \eqref{0106.9} 
to obtain a longer lifespan for the solution.

\begin{lemma} \label{lemma:induz}
Let 
\begin{equation*} 
0 < c_0 \leq 1, \quad \ 
0 < \rho_0 \leq \frac{\d_1}{2}, \quad \ 
x_0 = \frac{ K_* \rho_0^2 }{ c_0^2 } \leq \frac{1}{24},
\end{equation*}
where $\d_1$ is given in Lemma \ref{lemma:0106.3}
and $K_*$ in Proposition \ref{key}.
Let $N$ be an integer such that 
\begin{equation} \label{ind.N}
1 \leq N \leq 
\frac{\log 2}{| \log(1 - 12 x_0)|}\,.
\end{equation}
Let $(u_0,v_0) \in H^{m_1}_0(\T^d,c.c.)$, with
\begin{equation*} 
\| u_0 \|_{m_1} \leq \rho_0,
\end{equation*}
and assume that $u_0$ satisfies the nonresonance condition \eqref{hyp.nonres},
namely, with $\om_{\a\b\lm}$, $\Om_{\a\b\lm}$ defined in \eqref{def om},  
\begin{equation} \label{ind.nonres.0}
|\om_{\a\b\lm}| \geq c_0 \Om_{\a\b\lm} \quad \ \forall (\a,\b,\lm) \in \mT_1,
\end{equation}
where 
$\mT_1 := \{ (\a,\b,\lm) : \a, \b, \lm \in \Gamma_1, \ \a + \b = \lm \}$. 
Let $t_0 := 0$. 

Then for all $k = 1, \ldots, N$ the following properties hold.

\begin{itemize}
\item[$(i)_k$] The solution $u$ of system \eqref{0106.4} 
is defined on the interval $[0,t_k]$, where 
\begin{equation} \label{ind.t.tau}
t_k := t_{k-1} + \tau_k, 
\quad \ 
\tau_k := A_* \frac{c_{k-1}}{\rho_{k-1}^4},
\end{equation}
with $A_*$ given by Proposition \ref{key}. 
Moreover 
\begin{alignat}{2}
\label{ind.u}
\| u(t) \|_{m_1} 
& \leq \rho_k 
& \quad \ & \forall t \in [t_{k-1}, t_k],
\\
|\om_{\a\b\lm}(t)| 
& \geq c_k \Om_{\a\b\lm}(t) 
& \quad \ & \forall t \in [t_{k-1}, t_k], 
\ (\a,\b,\lm) \in \mT_1,
\label{ind.nonres.t}
\end{alignat}
where 
\begin{equation} \label{ind.rhok.ck}
\rho_k := \rho_{k-1} (1 + x_{k-1}), 
\quad \ 
c_k := c_{k-1} (1 - x_{k-1}).
\end{equation}
Also define 
\begin{equation} \label{ind.xk}
x_k := \frac{K_* \rho_k^2}{c_k^2}\,.
\end{equation}

\item[$(ii)_k$] 
One has 
\begin{align} 
\label{ind.est.x}
& 0 \leq x_k 
\leq x_0 (1 + 10 x_0)^k 
\leq 2 x_0 
\leq \frac{1}{12},
\\
\label{ind.est.c}
& 0 < \frac{c_0}{2} 
\leq c_0 (1 - 2 x_0)^k 
\leq c_k 
\leq c_0 
\leq 1,
\\
\label{ind.est.rho}
& 0 < \rho_0 
\leq \rho_k 
\leq \rho_0 (1 + 2 x_0)^k 
\leq 2 \rho_0 
\leq \d_1,
\\
\label{ind.est.tau}
& \frac{ c_k }{ \rho_k^4 } 
\geq \frac{ c_0 }{ \rho_0^4 } \, (1 - 12 x_0)^k
\geq \frac{ c_0 }{ 2 \rho_0^4 }.
\end{align}
\end{itemize}
\end{lemma}

\begin{proof} 
We start with proving the statements with $k=1$. 
The initial datum $u_0$ satisfies the assumptions of Proposition \ref{key}. 
We apply it, 
and we obtain that the solution $u$ is defined on $[0,T_*]$ 
with $T_* = A_* \rho_0^{-4} c_0$ (see \eqref{def.T2}), 
and $T_* = \tau_1 = t_1$ by definition \eqref{ind.t.tau}$|_{k=1}$. 
Also, by \eqref{1109.9}-\eqref{1106.4}, we get 
\begin{align*}
\| u(t) \|_{m_1} \leq \rho_0(1+x_0), \quad \ 
& \forall t \in [0,t_0],
\\
|\om_{\a\b\lm}(t)| \geq c_0(1-x_0) \Om_{\a\b\lm}(t) \quad \ 
& \forall t \in [0,t_0], \ \ (\a,\b,\lm) \in \mT_1,
\end{align*}
which are \eqref{ind.u}$|_{k=1}$, \eqref{ind.nonres.t}$|_{k=1}$. 
Hence $(i)_{k=1}$ is proved. 
By definition \eqref{ind.rhok.ck}$|_{k=1}$, \eqref{ind.xk}$|_{k=1}$, 
one has 
\begin{equation} \label{1306.1}
x_1 = x_0 \frac{(1+x_0)^2}{(1-x_0)^2}.
\end{equation}
We consider the elementary inequality 
\begin{equation} \label{elem.13}
\frac{(1+x)^2}{(1-x)^2} \leq 1 + \frac{4x}{(1-b)^2}
\quad \ \forall x, b \in \R, \ \ 
0 \leq x \leq b < 1,
\end{equation}
which holds because $x \mapsto \frac{ (1+x)^2 }{ (1-x)^2 }$ is convex on $[0,b]$, 
or just because 
\begin{align*}
\frac{(1+x)^2}{(1-x)^2} 
= \Big( 1 + \frac{2x}{1-x} \Big)^2 
\leq \Big( 1 + \frac{2x}{1-b} \Big)^2 
= 1 + x \Big( \frac{4}{1-b} + \frac{4x}{(1-b)^2} \Big)
\leq 1 + x \Big( \frac{4}{1-b} + \frac{4b}{(1-b)^2} \Big).
\end{align*}
For $b=\frac{1}{12}$ it implies that
\begin{equation} \label{elem.13.new}
\frac{(1+x)^2}{(1-x)^2} \leq 1 + 5x
\quad \ \forall x \in [0, \tfrac{1}{12}].
\end{equation}
Since $x_0 \leq \frac{1}{24}$, 
by \eqref{1306.1}, \eqref{elem.13.new} 
we get $x_1 \leq x_0 (1 + 5 x_0)$,
and \eqref{ind.est.x}$_{k=1}$ is satisfied.
Definition \eqref{ind.rhok.ck}$|_{k=1}$ gives 
$c_1 = c_0 (1-x_0)$, $\rho_1 = \rho_0 (1 + x_0)$,
and 
\eqref{ind.est.c}$|_{k=1}$, 
\eqref{ind.est.rho}$|_{k=1}$ follow immediately.
To prove \eqref{ind.est.tau}$|_{k=1}$, we consider the elementary inequality
\begin{equation} \label{elem.14}
\frac{1-x}{(1+x)^4} \geq 1 - (5 + 6b +4b^2 + b^3) x 
\quad \ \forall x,b \in \R, \ \ 0 \leq x \leq b,
\end{equation}
which holds because 
\[
\frac{1-x}{(1+x)^4} 
= 1 - \frac{5x + 6x^2 +4x^3 +x^4}{(1+x)^4}
\geq 1 - x (5 + 6x +4x^2 +x^3).
\]
For $b=\frac{1}{12}$ it implies that 
\begin{equation} \label{elem.14.new}
\frac{1-x}{(1+x)^4} \geq 1 - 6x 
\quad \ \forall x \in [0, \tfrac{1}{12}].
\end{equation}
Thus, by \eqref{elem.14.new}, 
\[
\frac{c_1}{\rho_1^4} 
= \frac{c_0 (1-x_0)}{\rho_0^4 (1+x_0)^4} 
\geq \frac{c_0}{\rho_0^4}\,(1 - 6 x_0),
\]
and \eqref{ind.est.tau}$|_{k=1}$ holds. 
This completes the proof of $(ii)_{k=1}$. 

\medskip

Now assume that $(i)_k, (ii)_k$ 
hold with $k=n$, 
for some $n \in [1,N-1]$; 
we have to prove them for $k=n+1$.
By \eqref{ind.u}$|_{k=n}$, \eqref{ind.nonres.t}$|_{k=n}$ one has
\[
\| u(t_n) \|_{m_1} \leq \rho_n, \quad \ 
|\om_{\a\b\lm}(t_n)| \geq c_n \Om_{\a\b\lm}(t_n).
\]
By \eqref{ind.est.c}$|_{k=n}$, \eqref{ind.est.rho}$|_{k=n}$, 
$0 < c_n \leq 1$, $0 < \rho_n \leq \d_1$. 
Hence Proposition \ref{key} can be applied with $(t_n, u(t_n))$ as initial data, 
and with $c_n, \rho_n$ in the role of the parameters $c_0, \rho$ of Proposition \ref{key}. 
We obtain that the solution is extended to the interval $[t_n, t_n + \tau_{n+1}]$, 
with $\tau_{n+1}$ given by Proposition \ref{key}, namely $\tau_{n+1} = A_* c_n \rho_n^{-4}$, 
which is also its definition in \eqref{ind.t.tau}$|_{k=n+1}$. 
With $\rho_{n+1}, c_{n+1}, x_{n+1}$ defined by  
\eqref{ind.rhok.ck}$|_{k=n+1}$, 
\eqref{ind.xk}$|_{k=n+1}$, 
Proposition \ref{key} also implies estimates 
\eqref{ind.u}$|_{k=n+1}$, \eqref{ind.nonres.t}$|_{k=n+1}$
on the time interval $[t_n, t_{n+1}]$. 
Thus $(i)_{k=n+1}$ is proved. 

Let us prove $(ii)_{k=n+1}$. 
One has $c_{n+1} = c_n (1 - x_n)$ by definition \eqref{ind.rhok.ck}$|_{k=n+1}$, 
$c_n \geq c_0 (1 - 2x_0)^n$ by \eqref{ind.est.c}$|_{k=n}$, 
and $x_n \leq 2x_0$ by \eqref{ind.est.x}$|_{k=n}$; therefore 
\[
c_{n+1} \geq c_0 (1 - 2 x_0)^{n+1}.
\]
By \eqref{ind.N} one also has $(1 - 2x_0)^N \geq \frac12$,
hence \eqref{ind.est.c}$|_{k=n+1}$ is proved. 

Similarly, $\rho_{n+1} = \rho_n (1 + x_n)$ by \eqref{ind.rhok.ck}$|_{k=n+1}$, 
$\rho_n \leq \rho_0 (1 + 2x_0)^n$ by \eqref{ind.est.rho}$|_{k=n}$, 
and $x_n \leq 2x_0$ by \eqref{ind.est.x}$|_{k=n}$; therefore 
\[
\rho_{n+1} \leq \rho_0 (1 + 2x_0)^{n+1}.
\]
By \eqref{ind.N} one also has $(1 + 2x_0)^N \leq 2$,
hence \eqref{ind.est.rho}$|_{k=n+1}$ is proved. 

From definitions \eqref{ind.xk}$|_{k=n+1}$, 
\eqref{ind.rhok.ck}$|_{k=n+1}$, 
\eqref{ind.xk}$|_{k=n}$ we deduce that  
\begin{equation} \label{x.recursive}
x_{n+1} = x_n \frac{(1+x_n)^2}{(1-x_n)^2}. 
\end{equation}
By \eqref{ind.est.x}$|_{k=n}$ we have 
$x_n \leq 2x_0 \leq \frac{1}{12}$. 
Hence, by \eqref{elem.13.new}, 
$x_{n+1} \leq x_n (1 + 5x_n)$.
Since $x_n \leq x_0 (1 + 10 x_0)^n$
and $x_n \leq 2x_0$ (both bounds coming from \eqref{ind.est.x}$|_{k=n}$), 
we obtain 
\[
x_{n+1} \leq x_0 (1 + 10 x_0)^{n+1}.
\]
By \eqref{ind.N} one also has $(1 + 10 x_0)^N \leq 2$, 
therefore \eqref{ind.est.x}$|_{k=n+1}$ is proved.

From definition \eqref{ind.rhok.ck}$|_{k=n+1}$ one has 
\[
\frac{c_{n+1}}{\rho_{n+1}} 
= \frac{c_n (1-x_n)}{\rho_n^4 (1+x_n)^4}.
\]
Since $x_n \leq 2x_0 \leq \frac{1}{12}$, 
by \eqref{elem.14.new} it follows that
\[
\frac{c_{n+1}}{\rho_{n+1}} 
\geq \frac{c_n}{\rho_n^4} \, (1 - 6x_n).
\]
Then we use \eqref{ind.est.tau}$|_{k=n}$ and the bound $x_n \leq 2x_0$, 
and obtain 
\[
\frac{c_{n+1}}{\rho_{n+1}} 
\geq \frac{c_0}{\rho_0^4} \, (1 - 12 x_0)^{n+1}.
\]
By \eqref{ind.N} one also has $(1 - 12 x_0)^N \geq \frac12$, 
therefore \eqref{ind.est.tau}$|_{k=n+1}$ is proved. 
The proof of $(ii)_{k=n+1}$ is complete.
\end{proof}

\begin{lemma} \label{lemma:incolla}
There exist universal constants $\d_3 \in (0,1)$, $A_3 > 0$ 
with the following properties. 
Let 
\[
0 < c_0 \leq 1, \quad \ 
0 < \e \leq \frac{\d_1}{2}, \quad \ 
\e \leq \d_3 c_0,
\]
where $\d_1$ is given in Lemma \ref{lemma:0106.3}.
Let $(u_0,v_0) \in H^{m_1}_0(\T^d,c.c.)$, with
\begin{equation} \label{ind.ball}
\| u_0 \|_{m_1} \leq \e,
\end{equation}
and assume that $u_0$ satisfies the nonresonance condition \eqref{hyp.nonres}. 
Then the solution $(u,v)$ of the Cauchy problem \eqref{0106.4} 
is defined on the interval $[0, T_3]$, where 
\[
T_3 = \frac{ A_3 c_0^3 }{ \e^6 }\,,
\]
with $(u,v) \in C([0,T_3], H^{m_1}_0(\T^d,c.c.))$ and   
\[
\| u(t) \|_{m_1} \leq 2 \e 
\quad \ \forall t \in [0,T_3].
\]
\end{lemma}

\begin{proof}
First, we consider the function 
\[
\ph(y) := \frac{y}{|\log(1-y)|}\,, 
\quad \ 0 < y \leq \frac12,
\]
we calculate its derivative 
\[
\ph'(y) = \frac{1}{|\log(1-y)|^2} \Big[ \log \Big(\frac{1}{1-y}\Big) - \frac{y}{1-y} \Big],
\]
and observe that $\ph'(y) < 0$ on $(0,\tfrac12]$ (apply the inequality $e^x > 1+x$ 
with $x = \frac{y}{1-y}$). 
Hence $\ph$ is decreasing, 
and therefore $\ph(y) \geq \ph(\frac12)$ for all $y \in (0,\frac12]$, 
namely 
\begin{equation} \label{elem.log}
\frac{\log 2}{|\log(1-y)|} \geq \frac{1}{2y}, 
\quad \ 0 < y \leq \frac12\,.
\end{equation}
As a consequence, given any real $0 < x \leq \frac{1}{48}$, 
there exists an integer $N$ such that 
\begin{equation} \label{N.x}
1 \leq \frac{1}{48 x} \leq N \leq \frac{1}{24 x} 
\leq \frac{\log 2}{|\log(1-12 x)|}
\end{equation}
(the interval $[\frac{1}{48x}, \frac{1}{24 x}]$ 
has length $\geq 1$, therefore it contains at least one integer).

Now let $0 < c_0 \leq 1$, 
$0 < \rho_0 \leq \frac12 \d_1$,
assume that $x_0 := K_* \rho_0^2 c_0^{-2} \leq \frac{1}{48}$, 
and let $N$ be an integer satisfying \eqref{N.x}$|_{x=x_0}$. 
Let $u_0$ satisfy \eqref{ind.ball}, \eqref{ind.nonres.0}. 
Then all the assumptions of Lemma \ref{lemma:induz} are satisfied. 
Thus the solution $u$ is defined on $[0,t_N]$, 
with 
\[
\| u(t) \|_{m_1} \leq 2 \rho_0 \quad \ \forall t \in [0, t_N]
\]
by \eqref{ind.u}, \eqref{ind.est.rho}, and 
\begin{align*}
t_N & = \sum_{k=1}^N \tau_k 
= A_* \sum_{k=0}^{N-1} \frac{c_k}{\rho_k^4}
\geq \frac{A_*}{2} \sum_{k=0}^{N-1} \frac{c_0}{\rho_0^4}
= \frac{A_* c_0}{2 \rho_0^4} N
\end{align*}
by \eqref{ind.t.tau}, \eqref{ind.est.tau}.
Then, by \eqref{N.x}$|_{x=x_0}$, 
\[
t_N \geq \frac{A_* c_0}{2 \rho_0^4} N
\geq \frac{A_* c_0}{2 \rho_0^4} \, \frac{1}{48 x_0} 
= \frac{A_3\, c_0^3}{\rho_0^6}
\]
with $A_3 := A_* (96 K_*)^{-1}$.
We define $\d_3 := (48 K_*)^{-\frac12}$, 
so that $x_0 \leq \frac{1}{48}$ becomes $\rho_0 \leq \d_3 c_0$,
and we rename $\e := \rho_0$.   
\end{proof}

\subsection{Back to the original coordinates}

We now aim at expressing the nonresonance condition \eqref{hyp.nonres} in the original coordinates.

Using the definition of the transformations $\Phi^{(3)}$, $\Phi^{(4)}$, $\Phi^{(5)}$ in \eqref{def Phi3}, \eqref{def Phi4}, \eqref{def Phi5} and reasoning like in the proof of Lemma \ref{lemma:W geq 7 new} (which is based on the structure described in Remark \ref{rem:matrix symbol}), one readily has the following.

\begin{lemma}\label{lemma:cism}
There exist universal constants $\d>0$, $C>0$ such that
for all $(u,v) \in H^{m_1}_0(\T^d, c.c.)$ in the ball $\| u \|_{m_1} \leq \d$, 
for all $k \in \Z^d$, 
the $k$-th Fourier coefficient $f_k=\overline{g_{-k}}$
of $(f,g):=(\Phi^{(3)}\circ\Phi^{(4)}\circ\Phi^{(5)})(u,v)$
satisfies
\begin{equation}\label{close id single mode}
|f_k-u_k| \leq C \| u \|_{m_1}^2 (|u_k| + |u_{-k}|)
\leq |u_k| + |u_{-k}|.
\end{equation}
An analogous bound holds for the inverse transformation, namely
\begin{equation*} 
|f_k-u_k| \leq C \| f \|_{m_1}^2 (|f_k| + |f_{-k}|)
\leq |f_k| + |f_{-k}|.
\end{equation*}
\end{lemma}

As a consequence, we have the following lemma on the superactions $S_\lm$.

\begin{lemma}\label{lemma:cisa}
There exist universal constants $\d>0$, $C>0$ such that
for all $(u,v) \in H^{m_1}_0(\T^d, c.c.)$ in the ball $\| u \|_{m_1} \leq \d$, 
letting $(f,g)=(\Phi^{(3)}\circ\Phi^{(4)}\circ\Phi^{(5)})(u,v)$ and denoting
\[
S_\lm = \sum_{|k|=\lm} |u_k|^2, \quad \tilde S_\lm = \sum_{|k|=\lm} |f_k|^2,
\]
one has for all $\lm\in\Gamma$
\begin{equation}\label{close id SA}
|\tilde S_\lm - S_\lm| \leq C \| u \|_{m_1}^2 S_\lm.
\end{equation}
An analogous bound holds for the inverse transformation, namely
\begin{equation}\label{close id SA inv}
|\tilde S_\lm - S_\lm| \leq C \| f \|_{m_1}^2 \tilde S_\lm.
\end{equation}
\end{lemma}

\begin{proof}
Let $\d>0$ be the same as in Lemma \ref{lemma:cism} and denote here by $\hat C>0$ the constant in \eqref{close id single mode}.
We start by observing (using Lemma \ref{lemma:cism}) that 
$|f_k| \leq 2(|u_k| + |u_{-k}|)$ and  
\begin{align*}
 \left| |f_k|^2 - |u_k|^2 \right| 
= (|f_k|+|u_k|) \left| |f_k| - |u_k| \right| 
& 
\leq 3 (|u_k| + |u_{-k}|) |f_k-u_k|
\\
& \leq 3 \hat C \| u \|_{m_1}^2 (|u_k| + |u_{-k}|)^2 
\\ & 
\leq 6 \hat C \| u \|_{m_1}^2 (|u_k|^2 + |u_{-k}|^2) .
\end{align*}
Hence, for all $\lm\in\Gamma$ one has
\begin{align*}
& |\tilde S_\lm - S_\lm| \leq \sum_{|k|=\lm} \left| |f_k|^2 - |u_k|^2 \right| \leq 12 \hat C \| u \|_{m_1}^2 S_\lm
\end{align*}
and \eqref{close id SA} holds with $C:=12\hat C$. In the same way one proves \eqref{close id SA inv}.
\end{proof}

From \eqref{close id SA}-\eqref{close id SA inv} it follows that 
$S_\lm = 0$ if and only if $\tilde S_\lm = 0$. 
Hence the set $\Gamma_1$ of the indices $\lm \in \Gamma$ for which $S_\lm$ is nonzero 
is left invariant by the transformation $(\Phi^{(3)}\circ\Phi^{(4)}\circ\Phi^{(5)})$.
We deduce the next lemma on the nonresonance condition.

\begin{lemma}\label{lemma:nonres}
Assume the hypotheses of Lemma \ref{lemma:cisa}.
If the sequence $(S_\lm)_{\lm\in\Gamma}$ satisfies \eqref{hyp.nonres} with some $c_0\in(0,1]$, then the sequence $(\tilde S_\lm)_{\lm\in\Gamma}$ satisfies
\begin{equation}\label{nonres perturb}
\big| \a^2 \tilde S_\a + \b^2 \tilde S_\b - \lm^2 \tilde S_\lm \big| 
\geq \left(c_0 - C\| u \|_{m_1}^2 \right) \big( \a^2 \tilde S_\a + \b^2 \tilde S_\b + \lm^2 \tilde S_\lm \big)
\end{equation}
for all $\a,\b,\lm\in\Gamma_1$ such that $\a+\b=\lm$.

The same statement applies to the inverse transformation $(\Phi^{(3)}\circ\Phi^{(4)}\circ\Phi^{(5)})^{-1}$.
\end{lemma}

\begin{proof}
We compute, applying Lemma \ref{lemma:cisa} and denoting by $\hat C$ 
the constant in \eqref{close id SA}, 
\begin{align*}
\big| \a^2 \tilde S_\a + \b^2 \tilde S_\b - \lm^2 \tilde S_\lm \big| 
& \geq \big| \a^2 S_\a + \b^2 S_\b - \lm^2 S_\lm \big| - \a^2 |\tilde S_\a - S_\a| - \b^2 |\tilde S_\b - S_\b| - \lm^2 |\tilde S_\lm - S_\lm|\\
& \geq (c_0 - \hat C \| u \|_{m_1}^2)
(\a^2 S_\a + \b^2 S_\b + \lm^2 S_\lm) 
\\
& \geq \frac{c_0 - \hat C \| u \|_{m_1}^2}{1 + \hat C \| u \|_{m_1}^2}   
(\a^2 \tilde S_\a + \b^2 \tilde S_\b + \lm^2 \tilde S_\lm)\\
& \geq \big( c_0 - 2 \hat C \| u \|_{m_1}^2 \big) 
(\a^2 \tilde S_\a + \b^2 \tilde S_\b + \lm^2 \tilde S_\lm),
\end{align*}
thus \eqref{nonres perturb} holds with $C=2\hat C$.
\end{proof}

Since the transformations $\Phi^{(1)}$, $\Phi^{(2)}$, defined in \eqref{1912.2}, \eqref{def Phi2}, are very explicit (the transformation $\Phi^{(1)}$ is only a Fourier multiplier and $\Phi^{(2)}$ leaves invariant the quantities $S_\lm$), we can now express the nonresonance condition \eqref{hyp.nonres} as a suitable condition on the datum in the original coordinates, by applying the normal form transformation 
$\Phi := \Phi^{(1)}\circ\Phi^{(2)}\circ\Phi^{(3)}\circ\Phi^{(4)}\circ\Phi^{(5)}$. 

To this end, given a pair of space-periodic real-valued functions $(a,b)$ 
as in \eqref{intro.ab}, 
we define the quantities 
$U_\lm := U_\lm(a,b)$ by \eqref{intro.def.U}. 
Lemma \ref{lemma:nonres} then translates immediately into the following.

\begin{lemma}\label{lemma:orig coord}
There exist universal constants $\d>0$, $C>0$ such that the following holds.
Let $(u,v) \in H^{m_1}_0(\T^d, c.c.)$ belong to the ball $\| u \|_{m_1} \leq \d$ 
and let $(a,b) = \Phi(u,v)$. 
If the sequence $\{ S_\lm = S_\lm(u) \}_{\lm\in\Gamma}$ 
satisfies \eqref{hyp.nonres} for some $c_0\in(0,1]$, 
then the sequence $\{ U_\lm = U_\lm(a,b) \}_{\lm\in\Gamma}$ satisfies
\begin{equation}\label{nonres orig}
|U_\a + U_\b - U_\lm| 
\geq \left(c_0 - C \| u \|_{m_1}^2 \right) (U_\a + U_\b + U_\lm)
\end{equation}
for all $\a,\b,\lm\in \Gamma_1$ such that $\a+\b=\lm$.
Conversely, if 
\begin{equation} \label{if.ab}
(a,b)\in H^{m_1+\frac12}_0(\T^d,\R) \times H^{m_1-\frac12}_0(\T^d,\R), 
\quad \ 
\| a \|_{m_1+\frac12} + \| b \|_{m_1-\frac12} \leq \d
\end{equation}
and, for some $c_0 \in (0,1]$, 
\begin{equation}\label{nonres dati Kir}
| U_\a + U_\b - U_\lm | 
\geq c_0 ( U_\a + U_\b + U_\lm )
\end{equation}
for all $\a,\b,\lm\in\Gamma_1$ such that $\a+\b=\lm$, 
then, setting $(u,v)=\Phi^{-1}(a,b)$, 
the sequence $\{ S_\lm = S_\lm(u) \}_{\lm\in\Gamma}$ satisfies 
\begin{equation}\label{bora bora}
\big| \a^2 S_\a + \b^2 S_\b - \lm^2 S_\lm \big| 
\geq \left(c_0 - C(\| a \|_{m_1+\frac12}^2 + \| b \|_{m_1-\frac12}^2) \right) 
\big( \a^2 S_\a + \b^2  S_\b + \lm^2 S_\lm \big)
\end{equation}
for all $\a,\b,\lm\in\Gamma_1$ such that $\a+\b=\lm$.
\end{lemma}

From Lemmas \ref{lemma:orig coord} and \ref{lemma:incolla} 
we deduce our main result on the Kirchhoff equation. 

\begin{proof}[\emph{\textbf{Proof of Theorem \ref{thm:main}}}]
Let $\e, c_0 \in (0,1]$, and assume that the datum $(a,b)$ 
satisfies \eqref{intro.ab.small}, \eqref{intro.nonres}. 
Let $(u_0, v_0) := \Phi^{-1}(a,b)$, where $\Phi = \Phi^{(1)} \circ \dots \circ \Phi^{(5)}$ 
is the normal form transformation. 
Hence 
\begin{equation} \label{dai}
\| u_0 \|_{m_1} \leq C_1 \e
\end{equation}
for some universal constant $C_1 > 0$. 

Denote $\hat \d, \hat C$ the universal constants of Lemma \ref{lemma:orig coord}. 
If $\e \leq \hat \d$, then $(a,b)$ satisfies \eqref{if.ab}, \eqref{nonres dati Kir}, 
and therefore, by Lemma \ref{lemma:orig coord}, 
the actions $S_\lm(u_0, v_0)$ satisfy \eqref{bora bora}. 
If $\hat C \e \leq \frac12 c_0$, then 
\[
\hat C (\| a \|_{m_1+\frac12}^2 + \| b \|_{m_1-\frac12}^2) 
\leq \hat C \e^2 
\leq \hat C \e 
\leq \frac{c_0}{2},
\]
and we obtain 
\begin{equation}\label{bora bora mezzi}
| \a^2 S_\a + \b^2 S_\b - \lm^2 S_\lm | 
\geq \frac{c_0}{2} ( \a^2 S_\a + \b^2  S_\b + \lm^2 S_\lm)
\end{equation}
for all $\a,\b,\lm \in \Gamma_1$, $\a+\b=\lm$. 

Now let $\tilde c_0 := \frac12 c_0$, $\tilde \e := C_1 \e$. 
By \eqref{dai}, \eqref{bora bora mezzi}, one has 
\[
0 < \tilde c_0 \leq 1, \quad \ 
0 < \tilde \e \leq \frac{\d_1}{2}, \quad \ 
\tilde \e \leq \d_3 \tilde c_0, \quad \ 
\| u_0 \|_{m_1} \leq \tilde \e
\]
if $\e \leq \frac{\d_1}{2 C_1}$, 
$\e \leq \frac{\d_3}{2C_1} c_0$, 
where $\d_1, \d_3$ are the universal constants in Lemma \ref{lemma:incolla}. 
Thus the assumptions of Lemma \ref{lemma:incolla} are satisfied, 
and we obtain that the solution $(u,v)$ of the Cauchy problem \eqref{0106.4}
is defined on $[0,T_3]$ with 
\[
T_3 = \frac{ A_3 \tilde c_0^3}{\tilde \e^6}, \quad \ 
\| u(t) \|_{m_1} \leq 2 \tilde \e \quad \ \forall t \in [0, T_3].
\]
Replacing $\tilde c_0 = \frac12 c_0$, $\tilde \e = C_1 \e$, 
we get  $T_3 = A_4 c_0^3 \e^{-6}$ for some universal constant $A_4$. 
Since $c_0 \leq 1$, all the conditions on $\e$ hold if 
$\e \leq \d c_0$, where we define 
$\d := \min \{ \hat \d, \frac{1}{2 \hat C}, \frac{\d_1}{2 C_1}, \frac{\d_3}{2 C_1} \}$, 
which is a universal positive constant. 
\end{proof}

\section{Appendix. Quasilinear normal form and transformations}
\label{sec:App A}

In this section we review the main formulas and inequalities 
of the normal form construction of \cite{K old}-\cite{K mem} 
(subsections \ref{sec:Lin}-\ref{sec:second step}),
then we derive the effective equations 
\eqref{3005.1}-\eqref{3105.8} 
and prove Lemma \ref{lemma:W geq 7 new} 
(subsection \ref{sec:App B}).

\subsection{Linear transformations} \label{sec:Lin}
The first two transformations $\Phi^{(1)}, \Phi^{(2)}$ in \cite{K old} are very standard,   
and transform system \eqref{p1} into another one (see \eqref{syst uv}) 
where the linear part is diagonal, 
preserving both the real and the Hamiltonian structure of the problem.
They are the symmetrization of the highest order 
and then the diagonalization of the linear terms.

\medskip

\emph{Symmetrization of the highest order.}
In the Sobolev spaces \eqref{def:Hs} of zero-mean functions, 
the Fourier multiplier 
\begin{equation*} 
\Lm := |D_x| : H^s_0 \to H^{s-1}_0, \quad  
e^{ij \cdot x} \mapsto |j| e^{ij \cdot x}
\end{equation*}
is invertible.   
System \eqref{p1} writes
\begin{equation} \label{1912.1}
\begin{cases} 
\pa_t u = v \\ 
\pa_t v = - ( 1 + \la \Lm u, \Lm u \ra ) \Lm^2 u,
\end{cases}
\end{equation}
where $\la \cdot , \cdot \ra$ is the real scalar product of $L^2(\T^d,\R)$; 
the Hamiltonian \eqref{K2} is
\[
H(u,v) = \frac12 \la v , v \ra + \frac12 \la \Lm u, \Lm u \ra 
+ \frac14 \la \Lm u, \Lm u \ra^2.
\]
To symmetrize the system at the highest order, 
we consider the linear, symplectic transformation 
\begin{equation} \label{1912.2}
(u,v) = \Phi^{(1)}(q,p) 
:= ( \Lm^{-\frac12} q , \Lm^{\frac12} p). 
\end{equation}
System \eqref{1912.1} becomes 
\begin{equation} \label{1912.3}
\begin{cases}
\pa_t q = \Lm p \\ 
\pa_t p = - ( 1 + \la \Lm^{\frac12} q, \Lm^{\frac12} q \ra ) \Lm q,
\end{cases}
\end{equation}
which is the Hamiltonian system $\pa_t (q,p) = J \gr H^{(1)}(q,p)$ 
with Hamiltonian $H^{(1)} = H \circ \Phi^{(1)}$, namely
\begin{equation} \label{def H(1)}
H^{(1)}(q,p) 
= \frac12 \la \Lm^{\frac12} p, \Lm^{\frac12} p \ra 
+ \frac12 \la \Lm^{\frac12} q, \Lm^{\frac12} q \ra 
+ \frac14 \la \Lm^{\frac12} q, \Lm^{\frac12} q \ra^2, 
\quad 
J := \begin{pmatrix} 
0 & I \\ 
- I & 0 \end{pmatrix}.
\end{equation}
The original problem requires the ``physical variables'' $(u,v)$ to be real-valued; 
this corresponds to $(q,p)$ being real-valued too.
Also note that $\la \Lm^{\frac12} p, \Lm^{\frac12} p \ra = \la \Lm p, p\ra$.

\medskip

\emph{Diagonalization of the highest order: complex variables.}
To diagonalize the linear part $\pa_t q = \Lm p$, $\pa_t p = - \Lm q$
of system \eqref{1912.3}, we introduce complex variables. 

System \eqref{1912.3} and the Hamiltonian $H^{(1)}(q,p)$ in \eqref{def H(1)} 
are also meaningful, without any change, for \emph{complex} functions $q,p$. 
Thus we define the change of complex variables $(q,p) = \Phi^{(2)}(f,g)$ as
\begin{equation} \label{def Phi2}
(q,p) = \Phi^{(2)}(f,g) := \Big( \frac{f+g}{\sqrt2}, \frac{f-g}{i \sqrt2} \Big),
\qquad 
f = \frac{q + i p}{\sqrt2}, \quad  
g = \frac{q - i p}{\sqrt2}, 
\end{equation}
so that system \eqref{1912.3} becomes
\begin{equation} \label{syst uv}
\begin{cases}
\pa_t f = - i \Lm f - i \frac14 \la \Lm(f+g) , f+g \ra \Lm(f+g)
\\
\pa_t g = i \Lm g + i \frac14 \la \Lm(f+g) , f+g \ra \Lm(f+g)
\end{cases}
\end{equation}
where the pairing $\la \cdot , \cdot \ra$ denotes the integral of the product 
of any two complex functions
\begin{equation} \label{SPWC} 
\la w , h \ra := \int_{\T^d} w(x) h(x) \, dx 
= \sum_{j \in \Z^d \setminus \{ 0 \} } w_j h_{-j}, 
\quad w, h \in L^2(\T^d, \C).
\end{equation}
The map $\Phi^{(2)} : (f,g) \mapsto (q,p)$ in \eqref{def Phi2} 
is a $\C$-linear isomorphism of the Sobolev space $H^s_0(\T^d,\C) \times H^s_0(\T^d,\C)$ 
of pairs of complex functions, for any $s \in \R$.  
When $(q,p)$ are real, $(f,g)$ are complex conjugate.
The restriction of $\Phi^{(2)}$ to the space
$H^s_0(\T^d,c.c.)$ (see \eqref{def H cc})
of pairs of complex conjugate functions 
is an $\R$-linear isomorphism onto the space $H^s_0(\T^d,\R) \times H^s_0(\T^d,\R)$ 
of pairs of real functions. 
For $g = \overline{f}$, the second equation in \eqref{syst uv} is redundant, 
being the complex conjugate of the first equation. 
In other words, system \eqref{syst uv} has the following ``real structure'': 
it is of the form 
\begin{equation*} 
\pa_t \begin{pmatrix} f \\ g \end{pmatrix} 
= \mF(f,g) = \begin{pmatrix} \mF_1(f,g) \\ \mF_2(f,g) \end{pmatrix}
\end{equation*}
where the vector field $\mF(f,g)$ satisfies 
\begin{equation} \label{real vector field}
\mF_2(f, \overline{f}) = \overline{ \mF_1(f, \overline{f}) }.
\end{equation}
Under the transformation $\Phi^{(2)}$, the Hamiltonian system \eqref{1912.3} 
for complex variables $(q,p)$ becomes \eqref{syst uv}, which is the Hamiltonian system 
$\pa_t(f,g) = i J \gr H^{(2)}(f,g)$ with Hamiltonian 
$H^{(2)} = H^{(1)} \circ \Phi^{(2)}$, namely
\begin{equation*} 
H^{(2)}(f,g) = \la \Lm f, g \ra + \frac{1}{16} \la \Lm(f+g), f+g \ra^2,
\end{equation*} 
where $J$ is defined in \eqref{def H(1)}, 
$\la \cdot , \cdot \ra$ is defined in \eqref{SPWC},
and $\gr H^{(2)}$ is the gradient with respect to $\la \cdot , \cdot \ra$.
System \eqref{1912.3} for real $(q,p)$ (which corresponds to the original Kirchhoff equation)
becomes system \eqref{syst uv} restricted to the subspace $H^s_0(\T^d,c.c.)$ where 
$g = \overline{f}$.

\subsection{Diagonalization of the order one}
In \cite{K old} 
the following nonlinear global transformation $\Phi^{(3)}$ is constructed.
Its effect is to remove the unbounded operator $\Lm$ from the ``off-diagonal'' terms of the equation, 
namely those terms coupling $f$ and $\overline f$.

\begin{lemma}[Lemma 3.1 of \cite{K old}] 
\label{lemma:Phi3 inv}
Let 
\begin{equation} \label{def Phi3}
\Phi^{(3)}(\eta,\psi) := \mM(\eta,\psi) \begin{pmatrix} \eta \\ \psi \end{pmatrix}, 
\end{equation}
where $\mM(\eta,\psi)$ is the matrix
\begin{equation*} 
\mM(\eta,\psi) := \frac{1}{\sqrt{1-\rho^2(P(\eta,\psi))}} 
\begin{pmatrix} 1 & \rho(P(\eta,\psi)) \\ 
\rho(P(\eta,\psi)) & 1 \end{pmatrix},
\end{equation*}
$\rho$ is the function
\begin{equation*} 
\rho(x) := \frac{- x}{1 + x + \sqrt{1+2x}}\,,
\end{equation*}
$P$ is the functional
\begin{equation*} 
P(\eta,\psi) 
:= \ph(Q(\eta,\psi)), \qquad 
Q(\eta,\psi) := \frac{1}{4} \la \Lm (\eta + \psi) , \eta + \psi \ra,
\end{equation*}
and $\ph$ is the inverse of the function $x \mapsto x \sqrt{1+2x}$, namely
\begin{equation*} 
x \sqrt{1+2x} = y \quad \Leftrightarrow \quad x = \ph(y).
\end{equation*}
Then, for all real $s \geq \frac12$, 
the nonlinear map $\Phi^{(3)} : H^s_0(\T^d, c.c.) \to H^s_0(\T^d, c.c.)$ 
is invertible, continuous, with continuous inverse 
\begin{equation*} 
(\Phi^{(3)})^{-1} (f,g) = \frac{1}{\sqrt{1 - \rho^2(Q(f,g))}} 
\begin{pmatrix} 1 & - \rho(Q(f,g)) \\ - \rho(Q(f,g)) & 1 \end{pmatrix} 
\begin{pmatrix} f \\ g \end{pmatrix}. 
\end{equation*}
For all $s \geq \frac12$, 
all $(\eta,\psi) \in H^s_0(\T^d,c.c.)$,  
one has 
\begin{equation*} 
\| \Phi^{(3)}(\eta,\psi) \|_s 
\leq C( \| \eta,\psi \|_{\frac12} ) \| \eta,\psi \|_s
\end{equation*}
for some increasing function $C$. 
The same estimate is satisfied by $(\Phi^{(3)})^{-1}$.
\end{lemma}

In \cite{K old} it is proved that system \eqref{syst uv}, 
under the change of variable $(f,g) = \Phi^{(3)}(\eta, \psi)$, 
becomes 
\begin{equation} \label{syst 6 dic}
\begin{cases}
\pa_t \eta = - i \sqrt{1 + 2 P(\eta,\psi)} \, \Lm \eta
+ \dfrac{i}{4(1+2 P(\eta,\psi))} 
\Big( \la \Lm \psi, \Lm \psi \ra - \la \Lm \eta , \Lm \eta \ra \Big) \psi,
\\
\pa_t \psi = i \sqrt{1 + 2 P(\eta,\psi)} \, \Lm \psi
+ \dfrac{i}{4(1+2 P(\eta,\psi))} 
\Big( \la \Lm \psi, \Lm \psi \ra - \la \Lm \eta , \Lm \eta \ra \Big) \eta.
\end{cases}
\end{equation}
System \eqref{syst 6 dic} is diagonal at the order one, 
i.e.\ the coupling of $\eta$ and $\psi$ (except for the coefficients) 
is confined to terms of order zero. 
Note that the coefficients of \eqref{syst 6 dic} are finite for $\eta,\psi \in H^1_0$, while the coefficients in \eqref{syst uv} are finite for $f,g \in H^{\frac12}_0$:  
the regularity threshold of the transformed system is $\frac12$ higher than before. 

The real structure is preserved, 
namely the second equation in \eqref{syst 6 dic} is the complex conjugate of the first one, 
or, in other words, the vector field in \eqref{syst 6 dic} satisfies property 
\eqref{real vector field}.
Even if $\Phi^{(3)}$ is not symplectic, nonetheless 
the transformed Hamiltonian $H^{(3)} := H^{(2)} \circ \Phi^{(3)}$
is still a prime integral of the equation, 
and it is 
\begin{align*}
H^{(3)}(\eta,\psi) 
& = \frac{- P(\eta,\psi) 
( \la \Lm \eta, \eta \ra + \la \Lm \psi, \psi \ra )}
{2\sqrt{1+2 P(\eta,\psi)}} 
+ \frac{1 + P(\eta,\psi)}{\sqrt{1+2 P(\eta,\psi)}} \la \Lm \eta, \psi \ra
+ P^2(\eta,\psi).
\end{align*}

As observed in \cite{K old}, since $P(\eta,\psi)$ is a function of time only (namely it does not depend on $x$), 
the vector field of \eqref{syst 6 dic} could be divided by a factor $\sqrt{1 + 2 P(\eta,\psi)}$
by a reparametrization of the time variable; this would normalize the terms of order one.
In \cite{K old}-\cite{K mem}, however, we did not make so, because it was not necessary.

\subsection{Normal form: first step} 
\label{sec:NF} 

The next step is the cancellation of the cubic terms contributing 
to the energy estimate. 
Following \cite{K old}, we write \eqref{syst 6 dic} as
\begin{equation} \label{syst DBR}
\pa_t (\eta,\psi) = X(\eta,\psi) := \mD_1(\eta,\psi) + \mD_{\geq 3}(\eta,\psi) 
+ \mB_3(\eta,\psi) + \mR_{\geq 5}(\eta,\psi)
\end{equation}
where 
\begin{equation*} 
\mD_1(\eta,\psi) := \begin{pmatrix} - i \Lm \eta \\ i \Lm \psi \end{pmatrix},
\quad 
\mD_{\geq 3}(\eta,\psi) := ( \sqrt{1 + 2 P(\eta,\psi)} \, - 1 ) \mD_1(\eta,\psi),
\end{equation*} 
$\mB_3(\eta,\psi)$ is the cubic component of the bounded, off-diagonal term
\begin{equation*} 
\mB_3(\eta,\psi) = 
\frac{i}{4} 
\Big( \la \Lm \psi, \Lm \psi \ra - \la \Lm \eta , \Lm \eta \ra \Big) 
\begin{pmatrix} \psi \\ \eta \end{pmatrix}
\end{equation*}
and $\mR_{\geq 5}(\eta,\psi)$ is the bounded remainder of higher homogeneity degree
\begin{equation} \label{def mR geq 5}
\mR_{\geq 5}(\eta,\psi) = 
\frac{- i P(\eta,\psi)}{2 (1 + 2 P(\eta,\psi))} 
\Big( \la \Lm \psi, \Lm \psi \ra - \la \Lm \eta , \Lm \eta \ra \Big) 
\begin{pmatrix} \psi \\ \eta \end{pmatrix}.
\end{equation}
The term $\mD_{\geq 3}$ gives no contribution to the energy estimate; 
the term $\mB_3$ is removed by the following normal form transformation.
Let 
\begin{equation} \label{def Phi4}
\Phi^{(4)}(w,z)
:= (I + M(w,z) ) \begin{pmatrix} w \\ z \end{pmatrix}, 
\end{equation}
\begin{equation} \label{def M} 
M(w,z) := \begin{pmatrix} 0 & A_{12}[w,w] + C_{12}[z,z] \\ 
A_{12}[z,z] + C_{12}[w,w] & 0 \end{pmatrix},
\end{equation}
where $A_{12}$, $C_{12}$ are the maps 
\begin{align}
\label{fix A12}
A_{12}[u, v] h
& := \sum_{j,k \neq 0, \, |j| \neq |k|} 
u_j v_{-j} \frac{|j|^2}{8(|j| - |k|)} h_k e^{ik \cdot x},    
\\
C_{12}[u, v] h 
& := \sum_{j,k \neq 0} u_j v_{-j} \frac{|j|^2}{8(|j| + |k|)} h_k e^{ik \cdot x}.
\label{fix C12}
\end{align}
Let
\begin{equation} \label{def m0}
m_0 := 1 \quad \text{if} \ d = 1, 
\qquad 
m_0 := \frac32 \quad \text{if} \ d \geq 2.
\end{equation} 

\begin{lemma}[Lemma 4.1 of \cite{K old}]
\label{lemma:stime AC}
Let $A_{12}, C_{12}, m_0$ be defined in \eqref{fix A12}, \eqref{fix C12}, \eqref{def m0}.
For all complex functions $u,v,h$, all real $s \geq 0$, 
\begin{equation*} 
\| A_{12}[u, v] h \|_s 
\leq \frac38 \| u \|_{m_0} \| v \|_{m_0} \| h \|_s,
\quad 
\| C_{12}[u, v] h \|_s 
\leq \frac{1}{16} \| u \|_1 \| v \|_1 \| h \|_s.
\end{equation*}
\end{lemma}

The differential of $\Phi^{(4)}$ at the point $(w,z)$ is 
\begin{equation*} 
(\Phi^{(4)})'(w,z) 
= (I + K(w,z)),  
\qquad 
K(w,z) = M(w,z) + E(w,z),
\end{equation*}
where $M(w,z)$ is defined in \eqref{def M}, and 
\begin{equation} \label{def E}
E(w,z) \begin{pmatrix} \a \\ \b \end{pmatrix}
:= \begin{pmatrix} 2 A_{12}[w,\a] z + 2 C_{12}[z,\b] z \\
2 C_{12}[w,\a] w + 2 A_{12}[z,\b] w \end{pmatrix}.
\end{equation}

\begin{lemma}[Lemma 4.2 of \cite{K old}]
\label{lemma:stime MK} 
For all $s \geq 0$, all $(w,z) \in H^{m_0}_0(\T^d, c.c.)$, 
$(\a,\b) \in H^s_0(\T^d,c.c.)$ one has
\begin{align*} 
\Big\| M(w,z) \begin{pmatrix} \a \\ \b \end{pmatrix} \Big\|_s 
& \leq \frac{7}{16} \| w \|_{m_0}^2 \|\a \|_s ,
\\
\Big\| K(w,z) \begin{pmatrix} \a \\ \b \end{pmatrix} \Big\|_s 
& \leq \frac{7}{16} \| w \|_{m_0}^2 \|\a \|_s 
+ \frac78 \| w \|_{m_0} \| w \|_s \| \a \|_{m_0} ,
\end{align*}
where $m_0$ is defined in \eqref{def m0}. 
For $\| w \|_{m_0} < \frac12$, 
the operator $(I + K(w,z)) : H^{m_0}_0(\T^d, c.c.)$ $\to H^{m_0}_0(\T^d, c.c.)$ 
is invertible, with inverse 
\[
(I + K(w,z))^{-1} = I - K(w,z) + \tilde K(w,z), \quad 
\tilde K(w,z) := \sum_{n=2}^\infty (- K(w,z))^n,
\]
satisfying
\begin{equation*} 
\Big\| (I + K(w,z))^{-1} \begin{pmatrix} \a \\ \b \end{pmatrix} \Big\|_s 
\leq C (\| \a \|_s + \| w \|_{m_0} \| w \|_s \| \a \|_{m_0}),
\end{equation*}
for all $s \geq 0$, where $C$ is a universal constant.
\end{lemma}

The nonlinear, continuous map $\Phi^{(4)}$ is invertible in a ball around the origin. 

\begin{lemma}[Lemma 4.3 of \cite{K old}] 
\label{lemma:Phi4 inv}
For all $(\eta, \psi) \in H^{m_0}_0(\T^d, c.c.)$ 
in the ball $\| \eta \|_{m_0} \leq \frac14$, 
there exists a unique $(w,z) \in H^{m_0}_0(\T^d, c.c.)$ such that 
$\Phi^{(4)}(w,z) = (\eta,\psi)$, with $\| w \|_{m_0} \leq 2 \| \eta \|_{m_0}$. 
If, in addition, $\eta \in H^s_0$ for some $s > m_0$, then $w$ also belongs to $H^s_0$, 
and $\| w \|_s \leq 2 \| \eta \|_s$.
This defines the continuous inverse map $(\Phi^{(4)})^{-1} : H^s_0(\T^d, c.c.) \cap 
\{ \| \eta \|_{m_0} \leq \frac14 \}$ $\to H^s_0(\T^d, c.c.)$.
\end{lemma}

Under the change of variables $(\eta,\psi) = \Phi^{(4)}(w,z)$,
it is proved in \cite{K old} that system \eqref{syst 6 dic} becomes 
\begin{align} 
\pa_t (w,z) 
& = (I + K(w,z))^{-1} X(\Phi^{(4)}(w,z))
=: X^+(w,z) 
\notag \\
& = \big( 1 + \mP(w,z) \big) \mD_1(w,z) + X_3^+(w,z)
+ X_{\geq 5}^+(w,z)
\label{def X+}
\end{align}
where
\begin{equation*} 
\mP(w,z) := \sqrt{1 + 2 P(\Phi^{(4)}(w,z))} \, - 1,
\end{equation*}
$X_3^+(w,z)$ has components
\begin{align*} 
(X_3^+)_1(w,z) 
& := - \frac{i}{4} \sum_{j,k \neq 0,\, |k| = |j|} w_j w_{-j} |j|^2 z_k e^{ik \cdot x},
\\ 
(X_3^+)_2(w,z) 
& := \frac{i}{4} \sum_{j,k \neq 0,\, |k| = |j|} z_j z_{-j} |j|^2 w_k e^{ik \cdot x},
\end{align*}
and
\begin{align*} 
X_{\geq 5}^+(w,z)
& := K(w,z) \big( I+K(w,z) \big)^{-1} \big( \mB_3(w,z) - X_3^+(w,z) \big)
+ \mR_{\geq 5}^+(w,z)
\notag \\ & \quad \ 
- \mP(w,z) \big( I+K(w,z) \big)^{-1} \big( \mB_3(w,z) - X_3^+(w,z) \big)
\end{align*}
with
\begin{align*}
\mR_{\geq 5}^+(w,z)
& := (I + K(w,z))^{-1} \mR_{\geq 5} ( \Phi^{(4)}(w,z))
+ [ \mB_3 ( \Phi^{(4)}(w,z)) - \mB_3(w,z) ]
\notag \\ & \qquad 
+ \big( - K(w,z) + \tilde K(w,z) \big) \mB_3 ( \Phi^{(4)}(w,z)),
\end{align*}
$\mR_{\geq 5}$ defined in \eqref{def mR geq 5}.

\begin{lemma}[Lemma 4.5 of \cite{K old}] 
\label{lemma:3001.2}
The maps $M(w, \overline{w})$, $K(w,\overline{w})$, 
and the transformation $\Phi^{(4)}$ 
preserve the structure of real vector field \eqref{real vector field}.
Hence $X^+$ defined in \eqref{def X+} satisfies \eqref{real vector field}.
\end{lemma}

The terms $(1+\mP) \mD_1$ and $X_3^+$ in \eqref{def X+}
give no contributions to the energy estimate, 
because, as one can check directly, 
\begin{equation*} 
\la \Lm^s (1 + \mP) (- i \Lm w) , \Lm^s z \ra 
+ \la \Lm^s w , \Lm^s (1 + \mP) i \Lm z \ra = 0
\end{equation*}
and
\begin{equation*}
\la \Lm^s (X_3^+)_1, \Lm^s z \ra + \la \Lm^s w , \Lm^s (X_3^+)_2 \ra 
= 0.
\end{equation*}
Similarly, also $\mP X_3^+$ gives no contribution to the energy estimate, 
because  
\[
\la \Lm^s (\mP X_3^+)_1, \Lm^s z \ra + \la \Lm^s w , \Lm^s (\mP X_3^+)_2 \ra 
= \mP \la \Lm^s (X_3^+)_1, \Lm^s z \ra 
+ \mP \la \Lm^s w , \Lm^s (X_3^+)_2 \ra = 0.
\]

\begin{lemma}[Lemma 4.6 of \cite{K old}] 
\label{lemma:elementary}
For all $s \geq 0$, all pairs of complex conjugate functions $(w,z)$, 
one has
\begin{equation*} 
\| \mB_3(w,z) \|_s \leq \frac12 \| w \|_1^2 \| w \|_s, 
\quad 
\| X_3^+(w,z) \|_s \leq \frac14 \| w \|_1^2 \| w \|_s, 
\end{equation*}
and, for $\| w \|_{m_0} \leq \frac12$, for all complex functions $h$,  
\begin{align} \label{stima mP}
\| \mP(w,z) h \|_s & = \mP(w,z) \| h \|_s,
\quad 
0 \leq \mP(w,z) \leq C \| w \|_{\frac12}^2,
\\ 
\| \mR_{\geq 5}(w,z) \|_s 
& \leq 2 P(w,z) \| \mB_3(w,z) \|_s 
\leq C \| w \|_{\frac12}^2 \| w \|_1^2 \| w \|_s
\notag
\end{align}
where $\mR_{\geq 5}$ is defined in \eqref{def mR geq 5}
and $C$ is a universal constant.
\end{lemma}

\begin{lemma}[Lemma 4.7 of \cite{K old}]  
\label{lemma:4.7 of K old} 
For all $s \geq 0$, all $(w,z) \in H^s_0(\T^d, c.c.) \cap 
H^{m_0}_0(\T^d, c.c.)$ with $\| w \|_{m_0} \leq \frac12$, one has 
\begin{equation} \label{stima X+ geq 5}
\| X_{\geq 5}^+ (w,z) \|_s 
\leq C \| w \|_1^2 \| w \|_{m_0}^2 \| w \|_s
\end{equation}
where $C$ is a universal constant.
\end{lemma}

In \cite{K mem} it is calculated that 
\begin{equation} \label{X+ new}
X^+(w,z) = (1 + \mP(w,z)) \big( \mD_1(w,z) + X_3^+(w,z) \big) 
+ X_5^+(w,z) + X_{\geq 7}^+(w,z)
\end{equation}
where $X_5^+(w,z)$ are terms of quintic homogeneity order 
extracted from $X^+_{\geq 5}(w,z)$, namely 
\begin{align} 
X^+_5(w,z) & := 
\mB_3'(w,z) M(w,z) \Big( \begin{matrix} w \\ z \end{matrix} \Big)
- K(w,z) X_3^+(w,z) 
- 3 Q(w,z) \mB_3(w,z),
\label{X+ 5}
\end{align}
and 
\begin{equation} \label{X+ 57}
X^+_{\geq 7}(w,z) := 
X^+_{\geq 5}(w,z) - \mP(w,z) X_3^+(w,z) - X^+_5(w,z).
\end{equation}
The terms $(1 + \mP(w,z)) ( \mD_1(w,z) + X_3^+(w,z) )$ in \eqref{X+ new} 
give no contributions to the energy estimate.  

\begin{lemma}[Lemma 4.8 of \cite{K mem}] 
\label{lemma:20.05.2020} 
For all $s \geq 0$, all $(w,z) \in H^s_0(\T^d, c.c.) \cap 
H^{m_0}_0(\T^d, c.c.)$ with $\| w \|_{m_0} \leq \frac12$, one has 
\begin{equation*} 
\| X_{5}^+ (w,z) \|_s 
\leq C \| w \|_{m_0}^4 \| w \|_s,
\quad \ 
\| X^+_{\geq 7}(w,z) \|_s \leq C \| w \|_{m_0}^6 \| w \|_s,
\end{equation*}
where $C$ is a universal constant.
\end{lemma}

It is calculated in \cite{K mem} that the first component $(X^+_5(w,z))_1$ 
of $X_5^+(w,z)$ is the sum of the following eight terms:
\begin{align*}
Y^{(4)}_{11}[w,w,w,w] w 
& := - \frac{i}{64} \sum_{j,\ell,k} 
\Big( \frac{|j|^2 |\ell|^2}{|j| + |k|} + \frac{|j|^2 |\ell|^2}{|\ell| + |k|} \Big) 
w_{j} w_{-j} w_{\ell} w_{-\ell} w_k e^{ik \cdot x},
\\
Y^{(2)}_{11}[w,w,z,z] w
& := \frac{i}{32} \sum_{j,\ell,k} |j|^2 |\ell|^2 
\Big( \frac{- \d_{|\ell|}^{|k|} (1 - \d_{|j|}^{|k|})}{|j| - |k|}
+ \frac{1}{|j|+|k|} 
- \frac{(1 - \d_{|\ell|}^{|k|})}{|\ell| - |k|} \Big) 
w_j w_{-j} z_\ell z_{-\ell} w_k e^{ik \cdot x},
\\
Y^{(0)}_{11}[z,z,z,z] w 
& := \frac{i}{64} \sum_{j,\ell,k} |j|^2 |\ell|^2 
\Big( - \frac{\d_{|\ell|}^{|k|} + \d_{|j|}^{|k|}}{|j| + |\ell|}
+ \frac{(1 - \d_{|j|}^{|k|})}{|j| - |k|} 
+ \frac{(1 - \d_{|\ell|}^{|k|})}{|\ell| - |k|} \Big) 
z_j z_{-j} z_\ell z_{-\ell} w_k e^{ik \cdot x},
\\
Y^{(4)}_{12} [w,w,w,w] z 
& := \frac{3 i}{32} \sum_{j,\ell,k} |j| |\ell| (|j| + |\ell|) 
w_j w_{-j} w_\ell w_{-\ell} z_k e^{ik \cdot x},
\\
Y^{(3)}_{12} [w,w,w,z] z
& := \frac{i}{16} \sum_{j,\ell,k} |j|^2 |\ell|
\Big( \frac{|\ell| \d_{|\ell|}^{|j|} (1 - \d_{|\ell|}^{|k|}) }{|\ell| - |k|}
+ 6 
+ \frac{|\ell|}{|\ell| + |j|}
+ \frac{|\ell| (1 - \d_{|\ell|}^{|j|}) }{|\ell| - |j|} \Big) 
w_j w_{-j} w_\ell z_{-\ell} z_k e^{ik \cdot x},
\\
Y^{(2)}_{12} [w,w,z,z] z
& := \frac{3 i}{16} \sum_{j,\ell,k} |j| |\ell| 
(|j| - |\ell|) 
w_j w_{-j} z_\ell z_{-\ell} z_k e^{ik \cdot x},
\\
Y^{(1)}_{12} [w,z,z,z] z 
& := \frac{i}{16} \sum_{j,\ell,k} |j| |\ell|^2
\Big( \frac{- |j| \d_{|j|}^{|\ell|} }{|j| + |k|} 
- 6 + \frac{|j| (1 - \d_{|j|}^{|\ell|}) }{|\ell| - |j|}
- \frac{|j|}{|\ell| + |j|} \Big)  
w_j z_{-j} z_\ell z_{-\ell} z_k e^{ik \cdot x},
\\
Y^{(0)}_{12} [z,z,z,z] z 
& := - \frac{3i}{32} \sum_{j,\ell,k} |j| |\ell| (|j| + |\ell|)
z_j z_{-j} z_\ell z_{-\ell} z_k e^{ik \cdot x},
\end{align*}
where $\d_{|k|}^{|j|}$ is the Kronecker delta, 
and when a coefficient is a fraction of the type $0/0$, it must be taken as zero
(this notation just avoids writing sums with several different restrictions  
on the summation set).

The second component $(X^+_5(w,z))_2$ 
of $X_5^+(w,z)$ is deduced from the first one by the real structure 
\eqref{real vector field}.

\subsection{Normal form: second step}
\label{sec:second step}

In \cite{K mem} we define the transformation 
\begin{equation} \label{def Phi5}
\begin{pmatrix} w \\ z \end{pmatrix} 
= \Phi^{(5)}(u,v) 
:= (I + \mM(u,v)) \begin{pmatrix} u \\ v \end{pmatrix},
\end{equation}
where 
\begin{equation} \label{def mM}
\mM(u,v) = \mA[u,u,u,u] + \mB[u,u,u,v] 
+ \mC[u,u,v,v] + \mD[u,v,v,v] + \mF[v,v,v,v],
\end{equation}
\[
\mA[u,u,u,u] = 
\begin{pmatrix} 
\mA_{11}[u,u,u,u] & \mA_{12}[u,u,u,u] \\ 
\mA_{21}[u,u,u,u] & \mA_{22}[u,u,u,u] 
\end{pmatrix},
\]
and similarly for the other terms in \eqref{def mM}; 
also,  
\[
\mA_{11}[u^{(1)} , u^{(2)}, u^{(3)}, u^{(4)}] h
:= \sum_{j,\ell, k} u^{(1)}_j u^{(2)}_{-j} u^{(3)}_{\ell} u^{(4)}_{-\ell} h_k \, a_{11}(j,\ell,k)
\, e^{ik \cdot x}
\]
for all $u^{(1)}, \ldots, u^{(4)}, h$, 
so that $\mA_{11}$ is determined by the coefficients 
$a_{11}(j,\ell,k)$, 
and similarly for all the other operators.
In \cite{K mem} we calculate the coefficients of the normal form transformations, 
which are 
\begin{align} 
\label{def a11 ottobrata}
a_{11}(j,\ell,k) 
& := \frac{ |j|^2 |\ell|^2}{128 (|j| + |\ell|)}  
\Big( \frac{1}{|j| + |k|} + \frac{1}{|\ell| + |k|} \Big),
\\
\label{def b11 ottobrata}
b_{11}(j,\ell,k) 
& := 0, 
\\
\label{def c11 ottobrata}
c_{11}(j,\ell,k) 
& := \frac{1}{64} |j|^2 |\ell|^2 
\Big( \frac{- \d_{|\ell|}^{|k|} (1 - \d_{|j|}^{|k|})}{|j| - |k|}
+ \frac{1}{|j|+|k|} 
- \frac{(1 - \d_{|\ell|}^{|k|})}{|\ell| - |k|} \Big)
\frac{1 - \d_{|j|}^{|\ell|} }{|\ell| - |j|},
\\ 
\label{def d11 ottobrata}
d_{11}(j,\ell,k) 
& := 0, 
\\
\label{def f11 ottobrata}
f_{11}(j,\ell,k) 
& := \frac{1}{128} 
\Big( - \frac{\d_{|\ell|}^{|k|} + \d_{|j|}^{|k|}}{|j| + |\ell|}
+ \frac{(1 - \d_{|j|}^{|k|})}{|j| - |k|} 
+ \frac{(1 - \d_{|\ell|}^{|k|})}{|\ell| - |k|} \Big) \frac{|j|^2 |\ell|^2 }{|j|+|\ell|},
\\
\label{def a12 ottobrata}
a_{12}(j,\ell,k) 
& := \frac{3}{64} |j| |\ell| (|j| + |\ell|)
\frac{(1 - \d_{|k|}^{|j|+|\ell|} )}{|k| - |j| - |\ell|} ,
\\
\label{def b12 ottobrata}
b_{12}(j,\ell,k) 
& := \frac{|j|^2 |\ell|}{32} 
\Big( \frac{|\ell| \d_{|\ell|}^{|j|} (1 - \d_{|\ell|}^{|k|}) }{|\ell| - |k|}
+ 6
+ \frac{|\ell|}{|\ell| + |j|}
+ \frac{|\ell| (1 - \d_{|\ell|}^{|j|}) }{|\ell| - |j|} \Big) 
\frac{1 - \d_{|j|}^{|k|} }{|k| - |j|},
\\
\label{def c12 ottobrata}
c_{12}(j,\ell,k) 
& := \frac{3}{32} |j| |\ell| (|j| - |\ell|) 
\frac{1 - \d_{|k|}^{|j| - |\ell|} }{|k| - |j| + |\ell|},
\\
\label{def d12 ottobrata}
d_{12}(j,\ell,k) 
& := \frac{|j| |\ell|^2}{32 (|k| + |\ell|)}
\Big( \frac{- |j| \d_{|j|}^{|\ell|} }{|j| + |k|} 
- 6 + \frac{|j| (1 - \d_{|j|}^{|\ell|}) }{|\ell| - |j|}
- \frac{|j|}{|\ell| + |j|} \Big),
\\ 
\label{def f12 ottobrata}
f_{12}(j,\ell,k) 
& := - \frac{3 |j| |\ell| (|j| + |\ell|) }{64 (|k| + |j| + |\ell|)},
\end{align}
with the same meaning of $0/0$ as above.
The differential of $\Phi^{(5)}$ is 
\begin{equation} \label{def mK}
(\Phi^{(5)})'(u,v) = I + \mK(u,v), 
\quad \ 
\mK(u,v) = \mM(u,v) + \mE(u,v)
\end{equation}
where 
\begin{align}
& \mE(u,v) \Big( \begin{matrix} \a \\ \b \end{matrix} \Big)
:= \{ 
2 \mA[u, \a, u,u] 
+ 2 \mA[u,u,u,\a]
+ 2 \mB[u, \a, u,v]
\notag \\ & \quad \ 
+ \mB[u,u,\a,v]
+ \mB[u,u,u,\b]
+ 2 \mC[u, \a, v, v]
+ 2 \mC[u, u, v, \b]
+ \mD[\a, v, v, v]
\notag \\ & \quad \ 
+ \mD[u, \b, v, v]
+ 2 \mD[u, v, v, \b]
+ 2 \mF[v, \b, v, v]
+ 2 \mF[v, v, v, \b]
\} \Big( \begin{matrix} u \\ v \end{matrix} \Big). 
\label{def mE}
\end{align}
With the change of variable \eqref{def Phi5}, 
the transformed equation is
\begin{equation} \label{3101.16}
\pa_t (u,v) = W(u,v)
\end{equation}
where 
\begin{equation*} 
W(u,v) := \big( (\Phi^{(5)})'(u,v) \big)^{-1} X^+( \Phi^{(5)}(u,v) ).
\end{equation*}
Recalling \eqref{X+ new}, we decompose 
\begin{equation} \label{W decomp}
W(u,v) = \big( 1 + \mP(\Phi^{(5)}(u,v)) \big) \big( \mD_1(u,v) + X_3^+(u,v) \big) 
+ W_5(u,v) + W_{\geq 7}(u,v),
\end{equation}
where $(1 + \mP(\Phi^{(5)})) (\mD_1 + X_3^+)$ give no contribution to the energy estimate, 
\begin{equation*} 
W_5(u,v) := X_5^+(u,v) + \mD_1(\mM(u,v)[u,v]) - \mK(u,v) \mD_1(u,v)
\end{equation*}
and $W_{\geq 7}(u,v)$ is defined by difference and contains only terms 
of homogeneity at least seven in $(u,v)$.
The first component $(W_5)_1(u,v)$ of the vector field $W(u,v)$ is 
\begin{align}
(W_5)_1(u,v) 
& = \frac{i}{32} \sum_{\begin{subarray}{c} j,\ell,k \\ |j| = |\ell| \end{subarray}} 
u_j u_{-j} v_\ell v_{-\ell} u_k e^{ik \cdot x} 
|j|^2 |\ell|^2 
\Big( \frac{1}{|j|+|k|} - \frac{(1 - \d_{|\ell|}^{|k|})}{|\ell| - |k|} \Big)
\notag \\ & \quad 
+ \frac{3 i}{32} 
\sum_{\begin{subarray}{c} j,\ell,k \\ |k| = |j| + |\ell| \end{subarray}} 
u_j u_{-j} u_\ell u_{-\ell} v_k e^{ik \cdot x} 
|j| |\ell| |k|
\notag \\ & \quad 
+ \frac{i}{16}  \sum_{\begin{subarray}{c} j,\ell,k \\ |j| = |k| \end{subarray}} 
u_j u_{-j} u_\ell v_{-\ell} v_k e^{ik \cdot x} 
|j|^2 |\ell|
\Big( 6 + \frac{|\ell|}{|\ell| + |j|}
+ \frac{|\ell| (1 - \d_{|\ell|}^{|j|}) }{|\ell| - |j|} \Big)
\notag \\ & \quad 
+ \frac{3 i}{16} \sum_{\begin{subarray}{c} j,\ell,k \\ |k| = |j| - |\ell| \end{subarray}} 
u_j u_{-j} v_\ell v_{-\ell} v_k e^{ik \cdot x} |j| |\ell| |k|,
\label{W5 comp 1 after NF}
\end{align}
and its second component is 
\begin{align}
(W_5)_2(u,v) 
& = - \frac{i}{32} \sum_{\begin{subarray}{c} j,\ell,k \\ |j| = |\ell| \end{subarray}} 
v_j v_{-j} u_\ell u_{-\ell} v_k e^{ik \cdot x} 
|j|^2 |\ell|^2 
\Big( \frac{1}{|j|+|k|} - \frac{(1 - \d_{|\ell|}^{|k|})}{|\ell| - |k|} \Big)
\notag \\ & \quad 
- \frac{3 i}{32} 
\sum_{\begin{subarray}{c} j,\ell,k \\ |k| = |j| + |\ell| \end{subarray}} 
v_j v_{-j} v_\ell v_{-\ell} u_k e^{ik \cdot x} 
|j| |\ell| |k|
\notag \\ & \quad 
- \frac{i}{16}  \sum_{\begin{subarray}{c} j,\ell,k \\ |j| = |k| \end{subarray}} 
v_j v_{-j} v_\ell u_{-\ell} u_k e^{ik \cdot x} 
|j|^2 |\ell|
\Big( 6 + \frac{|\ell|}{|\ell| + |j|}
+ \frac{|\ell| (1 - \d_{|\ell|}^{|j|}) }{|\ell| - |j|} \Big)
\notag \\ & \quad 
- \frac{3 i}{16} \sum_{\begin{subarray}{c} j,\ell,k \\ |k| = |j| - |\ell| \end{subarray}} 
v_j v_{-j} u_\ell u_{-\ell} u_k e^{ik \cdot x} |j| |\ell| |k|.
\label{W5 comp 2 after NF}
\end{align}

\begin{lemma}[Lemma 5.1 of \cite{K mem}] 
\label{lemma:stima W5} 
For all $s \geq 0$, $(w,z) \in H^s_0(\T^d, c.c.) \cap 
H^{m_0}_0(\T^d, c.c.)$, one has 
\begin{equation*} 
\| W_{5}(u,v) \|_s 
\leq C \| u \|_{m_0}^4 \| u \|_s,
\end{equation*}
where $C$ is a universal constant.
\end{lemma}

By \eqref{3101.16}, \eqref{W decomp}, \eqref{W5 comp 1 after NF}-\eqref{W5 comp 2 after NF}, 
the system for the Fourier coefficients becomes
\begin{align}
\pa_t u_k 
& = - i (1 + \mP) \Big( |k| u_k 
+ \frac{1}{4} \sum_{|j| = |k|} u_j u_{-j} |j|^2 v_k \Big)
\notag \\ & \quad 
+ \frac{i}{32} \sum_{\begin{subarray}{c} j,\ell \\ |j| = |\ell| \end{subarray}} 
u_j u_{-j} v_\ell v_{-\ell} u_k |j|^2 |\ell|^2 
\Big( \frac{1}{|j|+|k|} - \frac{(1 - \d_{|\ell|}^{|k|})}{|\ell| - |k|} \Big)
\notag \\ & \quad 
+ \frac{3 i}{32} 
\sum_{\begin{subarray}{c} j,\ell \\ |j| + |\ell| = |k| \end{subarray}} 
u_j u_{-j} u_\ell u_{-\ell} v_k |j| |\ell| |k|
\notag \\ & \quad 
+ \frac{i}{16} \sum_{\begin{subarray}{c} j,\ell \\ |j| = |k| \end{subarray}} 
u_j u_{-j} u_\ell v_{-\ell} v_k |j|^2 |\ell|
\Big( 6 + \frac{|\ell|}{|\ell| + |j|}
+ \frac{|\ell| (1 - \d_{|\ell|}^{|j|}) }{|\ell| - |j|} \Big)
\notag \intertext{} & \quad 
+ \frac{3 i}{16} \sum_{\begin{subarray}{c} j,\ell \\ |j| - |\ell| = |k| \end{subarray}} 
u_j u_{-j} v_\ell v_{-\ell} v_k |j| |\ell| |k| 
+ [(W_{\geq 7})_1(u,v)]_k
\label{eq for uk}
\end{align}
and
\begin{align}
\pa_t v_k 
& = i (1 + \mP) \Big( |k| v_k 
+ \frac{1}{4} \sum_{|j| = |k|} v_j v_{-j} |j|^2 u_k \Big)
\notag \\ & \quad 
- \frac{i}{32} \sum_{\begin{subarray}{c} j,\ell \\ |j| = |\ell| \end{subarray}} 
v_j v_{-j} u_\ell u_{-\ell} v_k |j|^2 |\ell|^2 
\Big( \frac{1}{|j|+|k|} - \frac{(1 - \d_{|\ell|}^{|k|})}{|\ell| - |k|} \Big)
\notag \\ & \quad 
- \frac{3 i}{32} 
\sum_{\begin{subarray}{c} j,\ell \\ |j| + |\ell| = |k| \end{subarray}} 
v_j v_{-j} v_\ell v_{-\ell} u_k |j| |\ell| |k|
\notag \\ & \quad 
- \frac{i}{16} \sum_{\begin{subarray}{c} j,\ell \\ |j| = |k| \end{subarray}} 
v_j v_{-j} v_\ell u_{-\ell} u_k |j|^2 |\ell|
\Big( 6 + \frac{|\ell|}{|\ell| + |j|}
+ \frac{|\ell| (1 - \d_{|\ell|}^{|j|}) }{|\ell| - |j|} \Big)
\notag \\ & \quad 
- \frac{3 i}{16} \sum_{\begin{subarray}{c} j,\ell \\ |j| - |\ell| = |k| \end{subarray}} 
v_j v_{-j} u_\ell u_{-\ell} u_k |j| |\ell| |k| 
+ [(W_{\geq 7})_2(u,v)]_k
\label{eq for vk}
\end{align}
where $[(W_{\geq 7})_1(u,v)]_k$ denotes the $k$-th Fourier coefficient 
of the first component of $W_{\geq 7}(u,v)$, and similarly for the second component; 
$\mP$ denotes, in short, $\mP(\Phi^{(5)}(u,v))$, which is a real function of time only, 
namely it is independent of $x$.

\begin{lemma}[Lemma 5.4 of \cite{K mem}] 
\label{lemma:ottobrata.4} 
For $d \geq 2$, 
the coefficients $a_{11}, c_{11}, f_{11}, a_{12}, b_{12}$, $c_{12}, d_{12}, f_{12}$ 
in \eqref{def a11 ottobrata}-\eqref{def f12 ottobrata} 
all satisfy the bound 
\[
| \text{coefficient}(k,j,\ell)| \leq C (|j|^4 |\ell|^2 + |j|^2 |\ell|^4)
\]
for some universal constant $C$. 
For $d=1$, they satisfy 
\[
| \text{coefficient}(k,j,\ell)| \leq C |j|^2 |\ell|^2.
\]
\end{lemma}

\begin{lemma}[Lemma 5.5 of \cite{K mem}] 
\label{lemma:ottobrata.5}
Let $m_1$ be defined in \eqref{def m1}.
All the operators $\mG \in \{ \mA_{11}, \mC_{11}, \mF_{11},$ 
$\mA_{12}, \mB_{12}, \mC_{12}, \mD_{12}, \mF_{12} \}$ satisfy 
\[ 
\| \mG[u, v, w, z] h \|_s 
\leq C \| u \|_{m_1} \| v \|_{m_1} \| w \|_{m_1} \| z \|_{m_1} \| h \|_s
\]
for all complex functions $u,v,w,z,h$, all real $s \geq 0$, 
where $C$ is a universal constant. 
\end{lemma}

\begin{lemma}[Lemma 5.6 of \cite{K mem}] 
\label{lemma:stime mM mK} 
For all $s \geq 0$, all $(u,v), (\a,\b)$, 
one has
\begin{align} \label{stima mM}
\Big\| \mM(u,v) \Big( \begin{matrix} \a \\ \b \end{matrix} \Big) \Big\|_s 
& \leq C \| u \|_{m_1}^4 \|\a \|_s ,
\\
\Big\| \mK(u,v) \Big( \begin{matrix} \a \\ \b \end{matrix} \Big) \Big\|_s 
& \leq C \| u \|_{m_1}^3 ( \| u \|_{m_1} \|\a \|_s + \| u \|_s \| \a \|_{m_1} ),
\notag
\end{align}
where $m_1$ is defined in \eqref{def m1} and $C$ is a universal constant.

There exists a universal constant $\d > 0$ such that, for $\| u \|_{m_1} < \d$, 
the operator $(I + \mK(u,v)) : H^{m_1}_0(\T^d, c.c.)$ $\to H^{m_1}_0(\T^d, c.c.)$ 
is invertible, with inverse satisfying
\begin{equation*} 
\Big\| (I + \mK(u,v))^{-1} \Big( \begin{matrix} \a \\ \b \end{matrix} \Big) \Big\|_s 
\leq C (\| \a \|_s + \| u \|_{m_1}^3 \| u \|_s \| \a \|_{m_1})
\end{equation*}
for all $s \geq 0$. 
\end{lemma}

The nonlinear, continuous map $\Phi^{(5)}$ is invertible in a ball around the origin. 

\begin{lemma}[Lemma 5.7 of \cite{K mem}] 
\label{lemma:Phi5 inv}
There exists a universal constant $\d > 0$ such that, 
for all $(w,z) \in H^{m_1}_0(\T^d, c.c.)$ in the ball $\| w \|_{m_1} \leq \d$, 
there exists a unique $(u,v) \in H^{m_1}_0(\T^d, c.c.)$ such that 
$\Phi^{(5)}(u,v) = (w,z)$, with $\| u \|_{m_1} \leq 2 \| w \|_{m_1}$. 
If, in addition, $w \in H^s_0$ for some $s > m_1$, then $u$ also belongs to $H^s_0$, 
and $\| u \|_s \leq 2 \| w \|_s$.
This defines the continuous inverse map $(\Phi^{(5)})^{-1} : H^s_0(\T^d, c.c.) \cap 
\{ \| w \|_{m_1} \leq \d \}$ $\to H^s_0(\T^d, c.c.)$.
\end{lemma}

By equations (5.35) of \cite{K mem}, 
the remainder $W_{\geq 7}(u,v)$ is given by
\begin{align} 
W_{\geq 7}(u,v) 
& = [1 + \mP(\Phi^{(5)}(u,v))] \breve \mK(u,v) (W_5(u,v) - X_5^+(u,v))
\notag \\ & \quad 
+ \mP(\Phi^{(5)}(u,v)) (W_5(u,v) - X_5^+(u,v))
\notag \\ & \quad 
+ \breve \mK(u,v) [1 + \mP(\Phi^{(5)}(u,v))] X_3^+(u,v) 
\notag \\ & \quad 
+ \breve \mK(u,v) X_5^+(u,v) 
\notag \\ & \quad 
+ (I + \mK(u,v))^{-1} [1 + \mP(\Phi^{(5)}(u,v))] [X_3^+(\Phi^{(5)}(u,v)) - X_3^+(u,v)]
\notag \\ & \quad 
+ (I + \mK(u,v))^{-1} [X_5^+(\Phi^{(5)}(u,v)) - X_5^+(u,v)]
\notag \\ & \quad 
+ (I + \mK(u,v))^{-1} X_{\geq 7}^+(\Phi^{(5)}(u,v)),
\label{W geq 7 bis}
\end{align}
where $\breve \mK(u,v) := \sum_{n=1}^\infty (- \mK(u,v))^n$.

\begin{lemma}[Lemma 5.8 of \cite{K mem}]
\label{lemma:stime W geq 7}
There exist universal constants $\d > 0$, $C > 0$ such that, 
for all $s \geq 0$, 
for all $(u,v) \in H^{m_1}_0(\T^d, c.c.) \cap H^s_0(\T^d, c.c.)$ 
in the ball $\| u \|_{m_1} \leq \d$, one has 
\begin{align*}
& \| W_{\geq 7}(u,v) \|_s \leq C \| u \|_{m_1}^6 \| u \|_s.
\end{align*}
\end{lemma}

\subsection{Derivation of the effective equation and structure of the remainder}
\label{sec:App B}

Now that the construction of the normal form has been recalled in details, 
to obtain the ``effective equations'' \eqref{3005.1}-\eqref{3005.5}
on Fourier spheres and to prove the estimates in Lemma \ref{lemma:W geq 7 new} 
for the single Fourier coefficient 
of the remainder $W_{\geq 7}(u,v)$ 
is not difficult. 

The derivation of \eqref{3005.1}-\eqref{3005.5}
is a straightforward calculation:
use the definition \eqref{def SB} of $S_\lm, B_\lm$, 
the equations \eqref{eq for uk}-\eqref{eq for vk} 
for the evolution of the Fourier coefficients $u_k, v_k$,
and sum over all indices $k \in \Z^d$ on the sphere $|k| = \lm$.

\begin{proof}[Proof of Lemma \ref{lemma:W geq 7 new}] 
The vector field $X(\eta,\psi)$ in \eqref{syst DBR}
is given by the simple explicit formula in \eqref{syst 6 dic}, 
where the multiplicative factors $P(\eta,\psi)$ and $(\la \Lm \psi, \Lm \psi \ra - \la \Lm \eta, \Lm \eta \ra)$ 
are functions of time, independent of $x$. 
Hence the Fourier coefficients of the remainder $\mR_{\geq 5}$ in \eqref{def mR geq 5} 
satisfy
\[
| [\mR_{\geq 5}(\eta,\psi)]_k| 
\leq \| \eta \|_{\frac12}^2 \| \eta \|_1^2 (|\eta_k| + |\eta_{-k}|)
\]
for all $(\eta,\psi) \in H^1_0(\T^d, c.c.)$, all $k \in \Z^d$
($|\psi_k| = |\eta_{-k}|$ because $\psi(x) = \overline{\eta(x)}$ 
and $\psi_k = \overline{(\eta_{-k})}$).
Recalling the definition \eqref{fix A12}-\eqref{fix C12} of $A_{12}, C_{12}$, 
and following the proof of Lemma 4.1 of \cite{K old}, 
we immediately obtain the inequalities for the Fourier coefficients
\[
|[ A_{12}[u,v]h ]_k| \leq \frac38 \| u \|_{m_0} \| v \|_{m_0} |h_k|, \quad \ 
|[ C_{12}[u,v]h ]_k| \leq \frac{1}{16} \| u \|_{1} \| v \|_{1} |h_k|
\]
for all complex-valued functions $u,v,h$, all $k$. 
Hence, from the definitions \eqref{def M}, \eqref{def E}, one has 
\begin{align*}
\Big| \Big[ M(w,z) \Big( \begin{matrix} \a \\ \b \end{matrix} \Big) \Big]_k \Big| 
& \leq C \| w \|_{m_0}^2 (|\a_k| + |\a_{-k}|),
\\
\Big| \Big[ E(w,z) \Big( \begin{matrix} \a \\ \b \end{matrix} \Big) \Big]_k \Big| 
& \leq C \| w \|_{m_0} \| \a \|_{m_0} (|w_k| + |w_{-k}|)
\end{align*}
for all $(w,z), (\a,\b) \in H^{m_0}_0(\T^d, c.c.)$, all $k \in \Z^d$. 
Applying recursively these bounds, using induction and Neumann series  
(like e.g.\ in the proof of Lemma 4.3 of \cite{K old}), 
we obtain estimates for the Fourier coefficients of the inverse operator 
\[
\Big| \Big[ (I + K(w,z))^{-1} \Big( \begin{matrix} \a \\ \b \end{matrix} \Big) \Big]_k \Big|
\leq C \{ (|\a_k| + |\a_{-k}|) + \| w \|_{m_0} \| \a \|_{m_0} (|w_k| + |w_{-k}|) \}
\]
for all $(w,z), (\a,\b) \in H^{m_0}_0(\T^d, c.c.)$, 
with $\| w \|_{m_0} \leq \d$, 
for all $k \in \Z^d$,
where $K(w,z) = M(w,z) + E(w,z)$ 
and $\d > 0$ is a universal constant. 
With (lengthly but straightforward) similar calculations, 
from formulas \eqref{X+ 5}, \eqref{X+ 57} for $X_5^+, X^+_{\geq 7}$ 
one proves that
\begin{align*}
| [X_5^+(w,z)]_k| 
& \leq C \| w \|_{m_0}^4 (|w_k| + |w_{-k}|), 
\\ 
| [X_{\geq 7}^+(w,z)]_k| 
& \leq C \| w \|_{m_0}^6 (|w_k| + |w_{-k}|)
\end{align*}
for all $(w,z) \in H^{m_0}_0(\T^d, c.c.)$, $\| w \|_{m_0} \leq \d$, for all $k \in \Z^d$. 

Then we repeat the same kind of (long, but simple and explicit) analysis 
for the operators $\mM(u,v), \mE(u,v), \mK(u,v)$ 
defined in \eqref{def mM}, \eqref{def mE}, \eqref{def mK}, 
and we estimate the Fourier coefficients of $W_{\geq 7}(u,v)$ 
using its formula \eqref{W geq 7 bis}.
\end{proof}

\begin{remark}\label{rem:matrix symbol}
The proof of Lemma \ref{lemma:W geq 7 new} 
is based on the properties of vector fields $V(u,v)$
having the structure $V(u,v) = F(u,v)(\begin{smallmatrix} u \\ v \end{smallmatrix})$
where $F(u,v)$ is a Fourier multiplier  
with matrix symbol depending (nonlinearly) on $(u,v)$. 
A more general version of Lemma \ref{lemma:W geq 7 new} 
for reality preserving transformed vector fields of this form 
can be proved with essentially the same ingredients. 
\end{remark}


\begin{footnotesize}

\end{footnotesize}

\bigskip

\begin{small}
Pietro Baldi

Dipartimento di Matematica e Applicazioni ``R. Caccioppoli''

Universit\`a di Napoli Federico II

Via Cintia, Monte S. Angelo, 
80126 Napoli

\texttt{pietro.baldi@unina.it}

\bigskip

Emanuele Haus

Dipartimento di Matematica e Fisica

Universit\`a Roma Tre

Largo San Leonardo Murialdo 1, 
00146 Roma

\texttt{ehaus@mat.uniroma3.it}
\end{small}

\end{document}